\documentclass[11pt]{amsart}
\usepackage[utf8]{inputenc}
\usepackage{fullpage,url,amssymb,enumitem,colonequals, mathrsfs,comment}
\usepackage{microtype}
\usepackage{xcolor}
\usepackage{tikz}
\usepackage{tikzit}
\usepackage[pdftex]{}

\usepackage[backref, colorlinks = true, linkcolor = blue, citecolor = blue]{hyperref}

\newcommand{\defi}[1]{\textsf{#1}} 

\newcommand{\cc}{\mathbb{C}}

\newcommand{\PP}{\mathbb{P}}
\newcommand{\pp}{\mathbb{P}}

\newcommand{\kbar}{{\bar{k}}}

\newcommand{\calO}{\mathcal{O}}
\renewcommand{\O}{\mathcal{O}}

\newcommand{\EE}{{\mathscr{E}}}

\DeclareMathOperator{\Char}{char}

\DeclareMathOperator{\Gal}{Gal}

\DeclareMathOperator{\Pic}{Pic}

\DeclareMathOperator{\supp}{supp}
\DeclareMathOperator{\Sym}{Sym}

\newcommand{\gon}{\operatorname{gon}}
\newcommand{\Span}{\operatorname{Span}}
\newcommand{\ev}{\operatorname{ev}}
\newcommand{\Ueno}{\operatorname{Ueno}}

\newcommand{\dendegs}{\delta}
\newcommand{\potdendegs}{\wp}

\newtheorem{theorem}{Theorem}[section]
\newtheorem{lemma}[theorem]{Lemma}
\newtheorem{corollary}[theorem]{Corollary}
\newtheorem{proposition}[theorem]{Proposition}

\theoremstyle{definition}
\newtheorem{definition}[theorem]{Definition}
\newtheorem{question}[theorem]{Question}
\newtheorem{mproblem}[theorem]{Main Problem}
\newtheorem{problem}[theorem]{Problem}
\newtheorem{setup}[theorem]{Setup}
\newtheorem{conjecture}[theorem]{Conjecture}

\theoremstyle{remark}
\newtheorem{remark}[theorem]{Remark}

\usepackage{microtype}
\title{Subspace configurations and low degree points on curves}

\author{Borys Kadets}
\address{Borys Kadets, Einstein Institute of Mathematics\\ Hebrew University of Jerusalem\\ Israel}
\email{kadets.math@gmail.com}

\urladdr{\url{http://bkadets.github.io}}

\author{Isabel Vogt}
\address{Isabel Vogt, Department of Mathematics \\ Brown University \\ Providence, RI 02906}
\email{ivogt.math@gmail.com}
\urladdr{\url{https://www.math.brown.edu/ivogt}}

\date{\today}

\begin{document}
\begin{abstract}
    This paper is devoted to understanding curves $X$ over a number field $k$ that possess infinitely many solutions in extensions of $k$ of degree at most $d$; such solutions are the titular low degree points. For $d=2,3$ it is known (\cite{Harris-Silverman1991}, \cite{Abramovich-Harris1991}) that such curves, after a base change to $\overline{k},$ admit a map of degree at most $d$ onto $\PP^1$ or an elliptic curve. For $d \geqslant 4$ the analogous statement was shown to be false \cite{Debarre-Fahlaoui1993}. We prove that once the genus of $X$ is high enough, the low degree points still have geometric origin: they can be obtained as pullbacks of low degree points from a lower genus curve. We introduce a discrete-geometric invariant attached to such curves: a family of subspace configurations, with many interesting properties. This structure gives a natural alternative construction of curves from \cite{Debarre-Fahlaoui1993}. As an application of our methods, we obtain a classification of such curves over $k$ for $d=2,3$, and a classification over $\kbar$ for $d=4,5$.
\end{abstract}
\maketitle

\section{Introduction}

Suppose $k$ is a number field and $X/k$ is a \defi{nice} curve (smooth, projective, and geometrically integral).  
The \defi{density degree set} \(\dendegs(X/k)\) of \(X/k\) is the set of integers \(d\) for which the collection of closed points of degree \(d\) on \(X\) are Zariski dense.  Since \(X\) is a curve, this is equivalent to asking that the degree \(d\) points be infinite, yet the definition in terms of Zariski density is natural for a variety of any dimension.  

In this paper we are concerned with the most basic such piece of information: the \defi{minimum density degree}\footnote{In previous work \cite{Smith-Vogt2020}, the second author called this invariant the \defi{arithmetic degree of irrationality}.  We prefer to switch to terminology that better generalizes to higher dimensional \(X\).} \(\min(\dendegs(X/k))\) is the smallest positive integer in \(\dendegs(X/k)\).
There is also a geometric version of the minimal density degree that is stable under finite extensions of the ground field.  Let \(\potdendegs(X/k)\) be the union of \(\dendegs(X/L)\) as \(L\) ranges over all finite extensions of \(k\). 
The \defi{minimum potential density degree} is \(\min(\potdendegs(X/k))\).  For more on the structure of \(\dendegs(X/k)\) and \(\potdendegs(X/k)\) see \cite{isolated_expository}.
Our motivating problem is:

\begin{mproblem}\label{main-prob}
Classify curves $X$ with $\min(\dendegs(X/k)) = d$, and those with \(\min(\potdendegs(X/k)) = d\).
\end{mproblem}

Faltings' theorem classifies curves \(X\) with \(\min(\potdendegs(X/k)) =1\).  Main problem \ref{main-prob} can therefore be viewed as a generalization of this fundamental problem.

There are two natural geometric sources of Zariski dense degree \(d\) points: if $X$ is a degree $d$ cover of $\PP^1$ or an elliptic curve $E$ of positive rank, then pulling back rational points on $\PP^1$ or $E$ gives an infinite family of degree $d$ (or less) points on $X$.
Previous work on Main Problem~\ref{main-prob} has focused on the geometric invariant \(\min(\potdendegs)\).
Harris--Silverman (for \(d=2\)) and Abramovich--Harris (for \(d=3\)) showed that the above two natural geometric sources of low degree points characterize when \(\min(\potdendegs)\) takes the value \(2\) or \(3\). More precisely, \(\min(\potdendegs(X/k)) = 2\) or \(3\) if and only if \(X_\kbar\) is a degree \(2\) or \(3\) cover of \(\pp^1\) or an elliptic curve.  Based on this evidence, Abramovich--Harris conjectured that the same should hold for all values of \(d\).  However, Debarre and Fahlaoui \cite{Debarre-Fahlaoui1993} showed that more obscure constructions of infinite families of degree \(d \geqslant 4\) points exist by cleverly constructing certain curves on the symmetric square of an elliptic curve.  A full classification for any larger values of \(d \geqslant 4\) has remained stubbornly out of reach, and there have been essentially no classification results for the arithmetic invariant \(\min(\dendegs)\).

In the present paper we refocus on Main Problem \ref{main-prob}, including the thornier arithmetic classification.  As a result, we obtain the following new classification.

\begin{theorem}\label{main-low-degrees}
Suppose $X/k$ is a nice curve. Then the following statements hold:
\begin{enumerate}
    \item If $\min(\dendegs(X/k)) = 2$, then $X$ is a double cover of $\PP^1$ or an elliptic curve of positive rank;
    \item If $\min(\dendegs(X/k)) =3$, then one of the following three cases holds:
    \begin{enumerate}
        \item  $X$ is a triple cover of $\PP^1$ or an elliptic curve of positive rank;
        \item\label{intro:case:quartic} $X$ is a smooth plane quartic with no rational points, positive rank Jacobian, and at least one cubic point;
        \item $X$ is a genus $4$  Debarre--Fahlaoui curve  (see Section \ref{sec:Debarre-Fahlaoui} for a precise definition);
\end{enumerate}
    \item If $\min(\potdendegs(X/k)) =d\leqslant 3$, then $X_\kbar$ is a degree $d$ cover of $\PP^1$ or an elliptic curve; 
    \item If $\min(\potdendegs(X/k)) = d=4,5$, then either $X_\kbar$ is a Debarre-Fahlaoui curve, or $X_{\kbar}$ is a degree $d$ cover of $\PP^1$ or an elliptic curve.
\end{enumerate}
\end{theorem}

Suprisingly, the seemingly clever construction by Debarre--Fahlaoui of counterexamples to the conjecture of Abramovich--Harris arises perfectly naturally from our perspective.  The next open case is to classify curves of genus \(11\) with \(\min(\potdendegs(X/k)) = 6\), see Section \ref{sec:geometric-questions}. As can be seen from Case \ref{intro:case:quartic} of Theorem \ref{main-low-degrees}, there are certain arithmetic subtleties involved in the classification; some open questions concerning these subtleties are described in Section \ref{sec:arithmetic-questions}.

\medskip

The classification in Theorem \ref{main-low-degrees} is obtained from a systematic study of the possible infinite collections of degree \(d\) points.
Our guiding philosophy is that when \(d\) is small compared to the genus of \(X\), such infinite collections \emph{still} occur for good geometric reasons.  The first step in our analysis is thus to make this precise with the following genus bound, which reduces Main Problem \ref{main-prob} to finitely many genera for each value of \(d\).

\begin{theorem}\label{genus-bound}
Suppose $X/k$ is a nice curve of genus $g$ and $\min(\dendegs(X/k))= d$. Let $m \colonequals \lceil d/2\rceil -1$ and let $\varepsilon \colonequals 3d-1-6m<6$. Then one of the following holds:
\begin{enumerate}
    \item\label{case:d-nonminimal} There exists a nonconstant morphism of curves $\phi\colon X \to Y$ of degree at least $2$ such that $d  =\min(\dendegs(Y/k)) \cdot \deg \phi ;$
   \item\label{eq-genus-bound} The genus of $X$ is bounded \[g \leqslant \max \left(\frac{d(d-1)}{2}+1,\ 3m(m-1)+m\varepsilon \right).\]
\end{enumerate} 
\end{theorem}

In case \eqref{case:d-nonminimal} there is a clear source of low degree points on $X$: they can be obtained as pullbacks of low degree point on $Y$ under $\phi$.  There is a long history of genus bounds in problems related to Main Problem \ref{main-prob}; see  \cite{Vojta1991quadratic},  \cite{Vojta1992generalization},  \cite{Abramovich-Harris1991}, \cite{song-tucker2001}.  Theorem \ref{genus-bound} is indirectly claimed in \cite{Abramovich-Harris1991} by combining  \cite[Lemma 3]{Abramovich-Harris1991} with \cite[Theorem 2]{Abramovich-Harris1991}; however, the statement of \cite[Lemma 3]{Abramovich-Harris1991} has an error, and the proof of \cite[Theorem 2]{Abramovich-Harris1991} contains a gap. See the discussion in Section \ref{relation}.

Following the ideas introduced by Abramovich and Harris, we study the geometry of curves with $\min(\dendegs(X/k)) = d$ by studying the geometry of linear systems of the form $|nD|$ for degree \(d\) points \(D\). The Mordell--Lang conjecture ensures that these linear systems have positive dimension. An important step in proving Theorem \ref{genus-bound} is Theorem \ref{2D-is-birational}, which states that unless case \eqref{case:d-nonminimal} holds, the linear systems $|nD|$ are birational for most $D$ and $n \geqslant 2$.

With birationality proved, we can investigate finer questions concerning the geometry of the linear systems $|nD|$. We do so by equipping these linear systems with a discrete-geometric structure: there is an infinite family of multisecant planes within each of the projective spaces $|nD|$, which form an combinatorially interesting configuration. The presence of this additional structure allows us to prove the following finer classification of curves \(X\) with \(\min(\dendegs(X/k)) = d\). To formally state this classification we require a notion of a sufficiently general degree $d$ point $D$; this is rigorously defined in Section \ref{sec:Mordell-Lang}. 

\begin{theorem}\label{main-finer}
Suppose $X$ is a curve with $\min(\dendegs(X/k)) =d$. Let $m \colonequals \lceil d/2\rceil -1$ and let $\varepsilon \colonequals 3d-1-6m<6$. Then for a sufficiently general degree $d$ point $D$ one of the following holds:
\begin{enumerate}
    \item $\dim |2D|=1$, and $X$ is a degree $d$ cover of an elliptic curve of positive rank;
    \item $\dim |2D|\geqslant 2$, the associated map $X \to \PP^{|2D|}$ is not birational onto its image, and there exists a covering of curves $\phi \colon X \to Y$ of degree at least $2$ such that $d  =\min(\dendegs(Y/k)) \cdot \deg \phi ;$
    \item $\dim |2D|=2$, the associated map $X \to \PP^2$ is birational onto its image, and $X$ is one of the Debarre-Fahlaoui curves (see Section \ref{sec:Debarre-Fahlaoui} for the precise definition);
    \item $\dim |2D|>2$,  the associated map $X \to \PP^{|2D|}$ is birational onto its image, and the genus $g$ of $X$ satisfies \[g \leqslant \max\left(\frac{(d-1)(d-2)}{2} + 2, \ 3m(m-1)+m\varepsilon \right)\]
\end{enumerate}

\end{theorem}
The proof of Theorem \ref{main-finer} involves a detailed analysis of the configuration geometry in $|3D|$; it shows how the geometry of the linear systems $|nD|$ naturally gives rise to the Debarre-Fahlaoui examples. 

\begin{remark}
The results of \cite{Abramovich-Harris1991} imply that the gonality of a curve with $\min(\dendegs(X/k)) = d$ is at most $2d$; this fact was also independently observed by Frey \cite{frey1994}. One corollary of Theorem \ref{main-finer} is that the geometric gonality can equal $2d$ only for (geometric) degree $d$ covers of elliptic curves; and can equal $2d-1$ only for (geometrically) Debarre-Fahlaoui curves. See Remarks \ref{rem:gonality-1} and \ref{rem:gonality-2}. 
\end{remark}

The method of examining multisecant configurations can be applied to study the low degree points on special families of curves.  As a demonstration, in Section \ref{sec:special-curves} we prove that projective curves of large genus have finitely many sufficiently low degree points. The statement of this estimate uses the Castelnuovo function $\pi(d,r)$; we recall its definition in Section \ref{sec:special-curves}. The number $\pi(d,r)$ is an upper bound for the genus of a nondegenrate degree $d$ curve in $\PP^r$. When $r$ is fixed and $d$ is growing, $\pi(d,r) \sim d^2/(2r-2)$.

\begin{theorem}\label{thms:intro-special-curves}
Suppose $X \subset \PP^r$ is an irreducible (possibly singular) curve of degree $e$ and genus $g$. Suppose $X$ has infinitely many points of degree $d$ not contained in hyperplanes of $\PP^r$. Then 
\[g \leqslant \pi(e+2d, 2r+1).\]
\end{theorem}

There are many open questions concerning the geometry of curves with abundant low degree points, both of arithmetic and purely geometric nature. We survey these questions in Section \ref{sec:questions}.

\subsection{Relation to previous work} \label{relation}
 The first results on low degree points were obtained by Hindry \cite{Hindry1987}, who studied quadratic points on modular curves $X_0(p)$ and asked if in general a curve with infinitely many quadratic points is either hyperelliptic or bielliptic. Later, Faltings \cite{Faltings1991}, and Vojta \cite{Vojta1991quadratic} used diophantine approximation techniques to describe low degree points on curves of small gonality. The strongest of these results was obtained by Vojta, who showed that a degree \(s\) cover of \(\pp^1\) with infinitely many degree \(d\) points not contracted by the map to \(\pp^1\) has genus at most \(s(d-1) +1\) \cite{Vojta1992generalization}.

The resulting genus bound is sharp: if $E$ is an elliptic curve of positive rank, then a $(d,s)$-curve on $E \times \PP^1$ satisfies the conditions of the theorem and has genus $g=s(d-1) +1.$
The general question of describing curves with minimum potential density degree $d$ was first addressed in \cite{Harris-Silverman1991} in the case $d=2$ and in \cite{Abramovich-Harris1991} for $d=3$. Based on these results, Abramovich and Harris proposed the following conjecture, which was soon disproved by Debarre and Fahlaoui.
\begin{conjecture}[\cite{Abramovich-Harris1991}; proved for \(d=2\) \cite{Harris-Silverman1991}; proved for \(d=3\) \cite{Abramovich-Harris1991}; disproved for all \(d\geqslant 4\) \cite{Debarre-Fahlaoui1993}]
Suppose $\min(\potdendegs(X/k)) = d$. Then $X_\kbar$ has a degree $d$ map to $\PP^1$ or an elliptic curve.
\end{conjecture}
The presence of counterexamples makes it hard to analyze the minimum density degree for arbitrary curves; however, the methods used in \cite{Abramovich-Harris1991} can still be applied to certain classes of special curves. For example, Debarre and Klassen \cite{Debarre-Klassen1994} showed that a smooth plane curve $X$ of degree $d \geqslant 8$ has minimum density degree $d$ or $d-1$ corresponding to the cases $X(k)=\emptyset$ and $X(k) \neq \emptyset$ respectively. For a generalization to curves on other surfaces see \cite{Smith-Vogt2020}.

A word of warning is warranted concerning the work of Abramovich and Harris \cite{Abramovich-Harris1991}: as detailed in \cite{Debarre-Fahlaoui1993}, the paper contains several errors (including in Lemma 3, Lemma 6, Lemma 8, and Corollary 1), which, while not severe enough to make the main results false, can be misleading.
 Some corrections are described in \cite{Debarre-Fahlaoui1993}, however one of the main results -- \cite[Theorem 2]{Abramovich-Harris1991} --  should be considered unproved. For this reason, we give full proofs of several simple lemmas appearing in \cite{Abramovich-Harris1991} when we need them. 

Other related work includes the study of integral points of low degree \cite{levin2016}, generalization of the work of Vojta to covers of curves \cite{song-tucker2001}, results on low degree points for curves on product surfaces \cite{levin2007}.

\subsection{Structure of the paper}

In Section \ref{sec:Mordell-Lang} we describe how the Mordell--Lang conjecture gives geometric restrictions on the curves with $\min(\dendegs(X/k)) = d$; this observation was also used in \cite{Harris-Silverman1991}, \cite{Abramovich-Harris1991}.

In Section \ref{sec:Birationality} we prove a key technical result: the birationality of the linear systems $|nD|$ for $n \geqslant 2$ and for sufficiently general degree $d$ points $D$ on $X$. More precisely, we show that these linear systems are birational if there does not exist a cover $\phi \colon X \to Y$ of degree at least $2$ with $d=\min(\dendegs(Y/k)) \cdot \deg \phi.$

In Section \ref{sec:subspace-configurations} we enrich the linear systems $|nD|$ with additional discrete-geometric structure of a configuration of multisecant subspaces to $X$. The main properties of this structure are summarized in Section \ref{sec:construction-summary}. As an application, we prove Theorem \ref{genus-bound}.

In Section \ref{sec:Debarre-Fahlaoui} we use subspace configruations to relate geometry of special linear systems to the construction of Debarre--Fahlaoui; in particular we prove most of Theorem \ref{main-finer}.

In Section \ref{sec:Applications} we collect geometric corollaries of the results obtained so far. We prove Theorem \ref{thms:intro-special-curves}, finish the proof of Theorem \ref{main-finer}, and summarize what we know about curves with minimum density degree at most $5$ to prove Theorem \ref{main-low-degrees}.

In Section \ref{sec:questions} we collect some open questions on the geometry of curves with $\min(\dendegs(X/k)) = d$.

\subsection{Notation}

Throughout the paper $k$ will denote a fixed number field and $X/k$ will denote a nice curve.  Let \(\bar{k}\) be an algebraic closure of \(k\).  We write \(\bar{X} = X_{\bar{k}}\) for the base-change of \(X\) to \(\bar{k}\).  By a \defi{degree \(d\)} point on \(X\) we mean a closed point with residue field a degree \(d\) extension of \(k\).  Write \(\Sym^d X = X^d/\!\!/S_d\) for the \(d\)th symmetric power of the curve \(X\), and \(\Pic^d_X\) for the degree \(d\) component of the Picard scheme of \(X/k\).  We write \(\Pic^dX\) for the group of isomorphism classes of degree \(d\) line bundles on \(X/k\).  There is an inclusion \(\Pic^dX \subset \Pic^d_X(k) = \Pic^d_X(\kbar)^{\Gal(\kbar/k)}\), which need not be an equality if \(X(k)  = \emptyset\).

We write \(W_dX = W^0_dX\) for the image of the Abel--Jacobi map \(\Sym^d X \to \Pic^d_X\) sending an effective divisor of degree \(d\) to its linear equivalence class.  This is a Brill--Noether locus of \(\Pic^d_X\) (see \cite[Chapter IV \(\S\)3]{ACGH} for more details on Brill--Noether loci). The fiber of the Abel--Jacobi map over a line bundle \(L \in \Pic^d X\) is isomorphic to the \defi{complete linear system} \(|L| \simeq \pp H^0(X, L)\) of \(L\).  When \(\dim |L| = r\), such a linear system is called a \defi{\(g^r_d\)}.  The minimal value of \(d\) for which \(X\) has a (\(k\)-rational) \(g^1_d\) is called the \defi{gonality} of \(X\), denoted \(\gon(X)\).

\subsection*{Acknowledgements}
We thank Dan Abramovich, Juliette Bruce, Tangli Ge, Brendan Hassett, Eric Larson, Daniel Litt,  Alexander Petrov, and Bianca Viray for helpful conversations and communications.
We also thank the anonymous referees for a careful reading of our paper and helpful comments.
Finally, we express our admiration for the foundational papers of Harris--Silverman \cite{Harris-Silverman1991}, Abramovich--Harris \cite{Abramovich-Harris1991}, Debarre--Fahlaoui \cite{Debarre-Fahlaoui1993}, and Frey \cite{frey1994} that inspired our work on this subject.
This material is partly based upon work supported by the National Science Foundation under Grant No. DMS-1928930 while B.K. participated in a program hosted by the Mathematical Sciences Research Institute in Berkeley, California, during the Fall 2020 semester.  I.V.~was partially supported by a NSF MSPRF under DMS-1902743 and by DMS-2200655.

\section{Corollaries of Mordell--Lang}\label{sec:Mordell-Lang}

In this section we collect some corollaries of the Mordell--Lang conjecture, now a theorem of Faltings \cite[Theorem 1]{Faltings1991}, for the structure of rational points on subvarieties of the Picard scheme of a curve over \(k\).

Suppose that $\min(\dendegs(X/k)) =d$ and that the gonality $\gon(X)$ is strictly greater than $d$. Since $\min(\dendegs(X/k))=d$, the set $\Sym^d X (k)$ is infinite. Since $\gon(X) > d$, the Abel--Jacobi map $\Sym^d X (k) \to \Pic^d_X(k)$ is injective. Consequently, the rational points of the image $W_dX(k)$ are an infinite set.  Since \(W_dX\) is a subvariety of a torsor under an abelian variety, the Mordell Lang conjecture \cite[Theorem 1]{Faltings1991} implies that there exists a translate $A \subset W_d X$ of a positive-dimensional abelian subvariety $A^0 \subset \Pic^0_X$, such that $A(k)$ is Zariski dense in $A$.

By the semicontinuity theorem, there is an open dense locus in \(A\) consisting of points for which the corresponding (isomorophism class of) line bundle \([L]\) has the minimal achieved value of \(h^0(X, L)\).  (By virtue of the fact that \(A \subset W^0_dX\), this minimal value is positive.)  The fibers of \(\Sym^d X \to \Pic^d_X\) are Severi--Brauer varieties of dimension \(h^0(X, L)-1\).  Since \(\gon(X) > d \), any fiber of \(\Sym^d X \to \Pic^d_X\) having a rational point is necessarily a \(\pp_k^0\).  Since \(\Sym^d X(k) \neq \emptyset\), the minimal achieved value of \(h^0(X, L)\) must be \(1\).  Hence there is an open dense locus \(U \subset A\) over which \(\Sym^d X \to \Pic^d_X\) is an isomorphism.
In the case \(\dim A = 1\), applying the curve-to-projective extension theorem to the map \(U \to \Sym^d X\), we see that there is a canonical effective divisor associated to every point of \(A\) (and in particular \(U(k) = A(k)\)).  Moreover, even when \(\dim A > 1\), resolving the indeterminacy of the inverse \(U \to \Sym^d X\) (in a similar manner to the proof the Lang--Nishimura lemma as in \cite[Proposition A.6]{RY00}) proves that \(U(k) = A(k)\).

Since we will focus on this setup for the majority of the paper, we codify it in the following:

\begin{setup}\label{main_setup}
Let \(X\) be a nice curve with \(\min(\dendegs(X/k))  = d\) and \(\gon (X) > d\).  Write \(A \subset W_dX\) for an associated positive-dimensional abelian translate and \(U \subset A\) for the open dense locus over which \(\Sym^d X \to \Pic^d_X\) is an isomorphism.
\end{setup}

Since the rational points of \(A\) are Zariski dense, any nonempty open in \(A\) also has Zariski dense rational points.  
By a \defi{general} rational point of \(A\), we mean a rational point in an open dense subvariety of \(A\).  By a \defi{general effective divisor} \(D\) in \(A(k)=U(k)\), we mean a rational point in an open dense subvariety of \(U \subset A\) where there exists a unique effective divisor representing each isomorphism class of line bundle.  
A general effective divisor \(D\) in \(A(k)\) is a degree \(d\) point on \(X\).

Consider the incidence correspondence 
\begin{equation}\label{eqtn:incidence}
     I \colonequals \{(p, [D]) \in X \times U : p \in D\}.
\end{equation}
Then \(I\) necessarily dominates \(X\) via the first projection, and is a degree \(d\) cover of \(U\) via the second projection.  The dimension of \(I\) is therefore equal to the dimension of \(A\).  The (arithmetic) monodromy of this degree \(d\) cover is transitive, since a general effective divisor in \(U\) corresponds to a degree \(d\) point on \(X\).  In particular, we have that
\begin{equation}\label{eq:disjoint_support}
\text{a general pair of effective divisors in \(U\) have disjoint support. }
\end{equation} 
The following lemma shows that at the same time, there are many pairs of divisors in \(U\) that share points on \(X\).

\begin{lemma}\label{two-divisors-through-a-point}
Suppose we are in Setup \ref{main_setup} and that if \(\dim A = 1\), then \(X\) is not a degree \(d\) cover of \(A\). Then for any open subset $V \subset U$ and a general point $P \in X(\kbar)$ there exists a pair of distinct divisors $D_1, D_2 \in V$ containing $P$.
\end{lemma}
\begin{proof}

Suppose to the contrary, that for some open $V \subset U$ a general point $P \in X$ is contained in a unique divisor from $V$. The rational map $\phi: X \dashrightarrow V$ that sends a point $P \in X$ to the unique divisor in $V$ that contains $P$ is dominant, and so $\dim A = 1$. In this case, the map $\phi$ extends to a degree $d$ covering $X \to A$ (by \eqref{eq:disjoint_support}).
\end{proof}

\subsection{The abelian translate property}\label{sec:abelian-translate}
The abelian translate \(A\) is a torsor under an abelian variety \(A^0 \subset \Pic^0_X\).  The group law on \(A^0\) has the following consequence, which we term the \defi{abelian translate property}: for any three points \(L_1\), \(L_2\) and \(L_3\) of \(A\), the line bundle \(L_1 \otimes L_2 \otimes L_3^{-1}\) is again a point of \(A\).  Moreover, on rational points, for effective divisors \(D_1\), \(D_2\) and \(D_3\) in \(U(k)\), 
the divisor \(D_1 + D_2 - D_3\) is again in \(U(k)\).

Fix an effective divisor $D$ in $U(k)$.  The abelian translate property implies that for any $n-1$ divisors $D_1, \dots, D_{n-1}$ in $U(k)$, there exists a $n$th divisor $D_n$ in $U(k)$ that is linearly equivalent to
\[nD - D_1 - \cdots - D_{n-1}.\]

In the next section, we will prove that Setup \ref{main_setup} implies that the linear systems \(|\O(nD)|\) for \(D\) a general effective divisor in \(A(k)\) are birational unless there is a natural geometric source of degree \(d\) points on \(X\).  We will then interpret the abelian translate property geometrically in terms of spans of divisors in \(|\O(nD)|\).

\section{Birationality}\label{sec:Birationality}
The main result of this section is Theorem \ref{2D-is-birational} which shows that in Setup \ref{main_setup}, unless an infinite collection of degree \(d\) points on \(X\) is obtained by pullback from a lower genus curve, the linear system $|2D|$ is birational for a general $D \in A(k)$. In particular, this immediately implies that the genera of such curves are bounded by \((d-1)(2d-1)\). These results are closely related to \cite{Abramovich-Harris1991}. Our main Theorem \ref{2D-is-birational} is  similar to \cite[Lemma 3]{Abramovich-Harris1991}; note, however, that the statement of \cite[Lemma 3]{Abramovich-Harris1991} has an error (the last formula of the statement is false), and more importantly the proof does not go through for curves which are degree $d$ covers of pointless conics -- the case that requires most work in our Theorem \ref{2D-is-birational}. 

We will use the following version of de Franchis theorem due to Kani.

\begin{theorem}\label{thm:kani}
Suppose $k$ is a field of characteristic zero and $X/k$ is a nice curve. Then
\begin{enumerate}
    \item There exist at most finitely many surjective morphisms $X \to Y$ to curves $Y/k$ of genus at least $2$;
    \item For any integer $d$ there exists at most finitely many surjective morphisms $X \to C$ of degree less than $d$ to curves of genus $1$, up to translations on the target.
\end{enumerate}
\end{theorem}
\begin{proof}
See \cite{Kani1986}{ Theorem~3} and \cite{Kani1986}{ Corollary~after~Theorem~4}.
\end{proof}

One way of obtaining infinitely many degree \(d\) points on \(X\) is via pullback from an elliptic curve of positive rank.  The following lemma describes a situation in which this is the case.

\begin{lemma}\label{lem:ramification_argument}
Assume that we are in Setup \ref{main_setup} and \(\dim A = 1\).  Let \(D \in A\) be a general divisor.  For every effective divisor $E \in A$, the divisor $E' \colonequals 2D-E$ belongs to $A$ by the abelian translate property. Suppose that there exists a map
\[\pi \colon X \to \pp^1\]
of degree \(2d\) such that all \(E \in A\) are contracted to a point \(\psi(E)\) by \(\pi\) and further \(\psi(E) = \psi(E')\).  Then there exists a degree \(d\) map \(\pi'\colon X \to A\) factoring the map \(\pi = \psi \circ \pi'\).
\end{lemma}
\begin{proof}
The morphism $\psi \colon A  \to \pp^1$ that sends a divisor $E \in A$ to the point $\pi(E)$ evidently factors through the quotient by the involution sending \(E\) to \(2D - E\).  By computing the ramification of \(\psi\), we will show that \(\psi\) has degree \(2\) and is hence equal to this quotient map.  As a result, the original map \(\pi\) factors \(X \xrightarrow{d:1} A \xrightarrow{\psi} \pp^1\), and \(X\) is a degree \(d\) cover of the elliptic curve \(A\).

Since \(\dim A = 1\), we can extend an inverse of the Abel--Jacobi map to a regular map \(A \to \Sym^d X\).  Since the effective divisor corresponding to a general point of \(A\) is reduced, the union of the supports of all nonreduced divisors from \(A\) is finite.  In particular, we may assume that \(D\), and all of its translates \(D'\) by the finitely many \(2\)-torsion points of \(A^0\), are disjoint from this finite set. 
The points of $D'$ are ramification points of $\pi$, and in particular we have the equality of sets $\pi^{-1}(\psi(D'))=\supp D'$. Since no nonreduced divisor intersects $\supp D' = \pi^{-1}(\psi(D'))$, any divisor from \(A\) supported on the fiber $\pi^{-1}(\psi(D'))$ is equal to $D'$. Therefore the map $\psi \colon A \to \pp^1$ is totally ramified over the $4$ points of the form $\psi(D')$ satisfying $2D'=2D$. Since $A$ has genus $1$ and $\psi$ is totally ramified over at least $4$ points, the Riemann--Hurwitz formula gives
    \[0=-2 \deg \psi + \sum_P (e_P -1) \geqslant -2 \deg \psi + 4(\deg \psi -1)=2 \deg \psi - 4.\]
    Therefore $\deg \psi =2$. This means that there are exactly two divisors from $A$ supported on a general fiber of $\pi$. 
    Recall the incidence correspondence $I\subset X \times U$ given by formula \eqref{eqtn:incidence} that represents the relation ``point belongs to a divisor''. Since we just saw that a general point of $X$ belongs to a unique divisor from $U$, the correspondence $I$ is a graph of a rational map $\pi':X \to A$. The map $\pi'$ represents the association $P \mapsto (\text{unique } D \in A\text{ with } P\in D)$ defined on an open dense subset of $X$, and gives the desired factorization $\pi=\psi \circ \pi'$.
\end{proof}

\begin{proposition}\label{prop:elliptic-exceptions} 
Suppose we are in Setup \ref{main_setup} and additionally that \(X\) is not a degree \(d\) cover of an elliptic curve.  Then for any divisor class $D \in A(k)$ the linear system \(|2D|\) is basepoint free and $\dim |2D| \geqslant \max(2, \dim A)$. 
\end{proposition}
\begin{proof}
By \cite[Lemma 1]{Abramovich-Harris1991}, \(|2D|\) is basepoint free and \(\dim |2D| \geqslant\dim A\).  (Indeed, by the abelian translate property, for all \(E \in A\), the class \(2D - E\) is effective.  The association \(E \mapsto E + |2D - E|\) defines a \((\geqslant \dim A)\)-dimensional family of effective divisors in class \(2D\). Since the divisors $E \in U$ don't have a shared point, the family of divisors of the form $E \cup [2D-E]$ do not have any common points either, and so $|2D|$is base point free.)  This completes the proof when \(\dim A > 1\), so we assume for the remainder of the proof that \(\dim A = 1\). 

Since $\dim |2D|$ is upper-semicontinuous, it suffices to prove the result for a general $D \in A$.  
Suppose that $\dim |2D|=1$, and let $\phi: X \to \PP^1$ be the associated map. For every effective divisor $E \in A$, the divisor $2D-E$ belongs to $A$ by the abelian translate property, and therefore is effective. Hence every divisor $E \in A$ is supported on a fiber of $\phi$ and both \(E\) and \(2D-E\) are supported on the same fiber.  Therefore the assumptions of Lemma \ref{lem:ramification_argument} are satisfied, which is a contradiction since we assumed that \(X\) is not a degree \(d\) cover of an elliptic curve.
\end{proof}
\begin{remark}\label{rem:gonality-1}
Abramovich and Harris, and, independently, Frey \cite{frey1994}, observed that the gonality of a curve with $\min(\dendegs(X/k)) = d$ is at most $2d$. An immediate corollary of Proposition \ref{prop:elliptic-exceptions} is that the geometric gonality of a curve with $\min(\dendegs(X/k)) = d$ that is not a degree $d$ cover of an elliptic curve is at most $2d-1$.
\end{remark} 
We now prove the main result of this section.
\begin{theorem}\label{2D-is-birational}
Suppose that we are in Setup \ref{main_setup} and \(D \in U(k)\) is a general divisor. Then one of the following holds:
\begin{enumerate}
    \item\label{non-birational} there exists a covering of curves $\phi \colon X \to Y$ of degree at least \(2\) with $\min(\dendegs(Y/k)) = d/ \deg \phi$;
        \item\label{birational} the associated map $X \to \PP^{\dim |2D|}$ is birational. (The basepoint free line bundle \(|2D|\) is birationally very ample).

\end{enumerate}
\end{theorem}
\begin{proof}    
By Theorem \ref{thm:kani}, $X$ has only finitely many nonconstant maps $f_1, ... , f_N$ of degree at most \(d\) to curves of genus $\geqslant 1$ up to automorphisms of the base. Since $D$ is general, we can assume that $D$ does not intersect the preimage of the branch locus of any of the $f_i$. 

Suppose that case \eqref{birational} does not hold, i.e., that the morphism $X \to \PP^{|2D|}$ factors as $X \xrightarrow{f} Y \hookrightarrow \PP^{|2D|}$ with $m\colonequals \deg f \geqslant 2$. Then we will show that case \eqref{non-birational} holds.  Note that it suffices to show that \(\min(\dendegs(Y/k)) \leqslant d/\deg \phi\), since \(X\) has finitely many points of degree less than \(d\). Write $\widetilde{f}: X \to \widetilde{Y}$ for the map to the normalization of \(Y\).  Since the nondegenerate curve $Y \subset \PP^{|2D|}$ has degree at least $\dim |2D|\geqslant 2$, the degree of $\widetilde{f}$ is at most $d$. 

First suppose that the genus of $\widetilde{Y}$ is at least $1$. By assumption, the divisor $D$ has trivial intersection with the preimage of the branch locus of the map \(\widetilde{f}\). Observe that for any curve $C$ and any effective divisor $\Delta$ the following property holds: if for some positive $k$ the linear system $k\Delta$ is base-point free and $\Phi: C \to \PP^N$ is the associated morphism, then we have the equality of sets $\Phi^{-1}(\Phi (\Delta)) = \Delta$. Applying this to $\Delta=D$ and $k=2$ gives $\widetilde{f}^{-1}\left(\widetilde{f}(D)\right)=D$. Since $D$ does not intersect the preimage of the branch locus, we have
\[d = \#D = \#\widetilde{f}^{-1}\left(\widetilde{f}(D)\right)=m\# \widetilde{f}(D).\] 
Hence the image of $D$ in $\widetilde{Y}$ is a point of degree equal to $d/m$.  Since there are infinitely many choices of $D$ and only finitely many choices for the morphism $\widetilde{f}$, by Theorem \ref{thm:kani}, there exists a map $\widetilde{f}:X \to \widetilde{Y}$ of degree $m$  such that for infinitely many $D \in \Sym^d X(k)$ the image $\widetilde{f}(D)$ has degree $d/m$, in which case (\ref{non-birational}) holds. 

Now consider the case that $\widetilde{Y}$ has genus $0$.  Since a genus $0$ curve has infinitely many quadratic points either (\ref{non-birational}) holds, or $m = \deg \widetilde{f} > d/2$, and so $\deg Y=2d/m < 4$. If $\deg Y=3$ and the genus of $\widetilde{Y}$ is zero, then \(Y\) has odd degree points, and thus $Y=\PP^1$; this implies that case (\ref{non-birational}) holds. 
    
It therefore remains to consider the case $\deg Y=2$, $\dim |2D|=2$, $Y$ is a smooth pointless conic and $f$ is a covering of degree $d$. In this way, a general divisor $D \in U$ defines a map to a rational curve $f_D: X \to Y_D$; let $U' \subset U$ be the open subset of divisors that define such maps. For every pair $(D_1, D_2) \in U' \times U'$ we get a map $\phi_{D_1, D_2}=(f_{D_1}, f_{D_2}): X \to Y_{D_1}\times Y_{D_2}$. Let $Z_{D_1, D_2}$ denote its image. By semicontinutity, $Z_{D_1, D_2}$ will have a constant bidegree \((e_1,e_2)\) on an open dense subset of $U' \times U'$, and comparing the degrees of $Z_{D_1, D_2}$ and $Z_{D_2, D_1}$ we see that the bidegree is necessarily symmetric: $e_1=e_2=e$.  
To simplify notation, fix a general pair \((D_1, D_2) \in U' \times U'\) and write \(Y_1 = Y_{D_1}\), \(Y_2 = Y_{D_2}\), \(f_1 = f_{D_1}\), \(f_2 = f_{D_2}\), \(\phi = \phi_{D_1, D_2}\), and \(Z = Z_{D_1, D_2}\).
In what follows, we will show that \(e>1\), so that the map \(\phi \colon X \to Z\) has degree \(d/e \leqslant d/2\), and that the images \(\phi(E)\) of divisors \(E \in U(k)\) have low enough degree to force us to be in case \eqref{non-birational}.

    Since the line bundles \(\O_X(D_1) = f_{1}^*\O_{\pp^1}(1)\) and \(\O_X(D_2) = f_{2}^*\O_{\pp^1}(1)\) are distinct, there does not exist an automorphism of \(\pp^1\) bringing one to the other and thus \(Z\) cannot be a \((1,1)\)-divisor.  Hence, the degree \(e\) of the projection from $Z$ to $Y_{1}$ and $Y_{2}$ is at least $2$. 
    For a general divisor $E\in U(k)$, the divisors \(2D_i - E\) are effective by the abelian translate property.  Therefore $f_i(E)$ is contained in the hyperplane section of a (pointless) conic and hence $\deg f_1(E)=\deg f_2(E)=2$, and so $\deg \phi(E) \leqslant 4$.  Since the map $f_1$ factors through $\phi$ and $\deg  f_1(E) =2$, the degree of $\phi(E)$ is either $2$ or $4$.

\smallskip
    
\noindent
\textbf{\boldmath Case 1: $\deg \phi(E)=2$ for infinitely many \(E \in U(k)\).} Since \(\deg(\phi) \leqslant d/2\), we are in evidently in case \eqref{non-birational}, unless \(d=2\) and the map \(\phi\) is birational onto its image.  However an integral \((2,2)\)-curve (which necessarily has geometric genus \(0\) or \(1\)) with infinitely many degree \(2\) points is always a degree \(2\) cover of \(\pp^1\), so we are again in case \eqref{non-birational}.
    
\smallskip
   
\noindent
\textbf{\boldmath Case 2: $\deg \phi(E)=4$ for general \(E \in U\).} In this case, we will show that \(e\geqslant 4\), and hence we are in case \eqref{non-birational} unless \(d=4\) and the map \(\phi\) is birational onto its image. Consider a general $E \in A$ and the divisors $f_2(E)$ and $f_2(2D_1-E)$ on $Y_2$. Each one is a degree $2$ point on $Y_2$, and as $D_1$ varies the point $f_2(2D_1-E)$ will vary as well. Since $D_1, D_2$ are a general pair, we can assume that the divisors $f_2(E)$ and $f_2(2D_1-E)$ are disjoint for general \(E \in A\), and hence the divisors $\phi(E), \phi(2D_1-E) \in Z$ are necessarily disjoint degree \(4\) divisors. Since $f_1(E)=f_1(2D_1-E)$, the projection
 of the degree \(8\) divisor \(\phi(E)+\phi(2D_1-E)\) to \(Y_1\) is supported on a degree $2$ point $f_1(E)$. Therefore the degree $e$ of the projection $Z \to Y_1$ is at least $4$, and hence the degree of $\phi: X \to Z$ is at most $\deg f_1 / 4 = d/4$. If $\phi$ is not birational onto its image, then since $\deg \phi \leqslant d/4$ and $Z$ has infinitely many degree $4$ points, case (\ref{non-birational}) holds. 
    
    It remains to consider the case when \(d=4\) and $X$ is birational to a $(4,4)$ curve $Z$ on $Y_1 \times Y_2$. If $Z$ is smooth, then $X=Z$ and the projections onto \(Y_1\) and \(Y_2\) give the only two degree \(4\) maps from $\bar{X}$ to \(\pp^1\) \cite[Chapter IV, Exercise F-2]{ACGH}.  This is a contradiction, since the infinite family of \(D_i\) define distinct maps. Therefore $Z$ is singular. Since $Y_1$ is pointless, the singular locus of $Z$ has to have cardinality $2$ or $\geqslant 4$, for otherwise the projection of the singular locus would be a zero-cycle of odd degree on $Y_1$. Therefore the genus of $X$ is either $7$ or at most $5$. Since $\dim |2D|=2$, the Riemann--Roch theorem implies that $g(X)=7$. Consider now the geometric curve $\bar{X}$. By Mumford's extension of Marten's theorem (see \cite[Chapter IV, Theorem 5.2]{ACGH}), since the curve $\bar{X}$ has a positive-dimensional family of \(g^1_4\)'s, it is either hyperelliptic, trigonal, bielliptic, or a smooth plane quintic.  Since \(g(X) = 7\), it is not isomorphic to a smooth plane quintic. If \(\bar{X}\) is hyperellitpic, then \(X\) is a degree \(2\) cover of a conic, which has infinitely many degree \(2\) points, and we are in case (1).  If $\bar{X}$ is trigonal, then the associated $(3,4)$ map onto $\PP^1\times Y_2$ has to be birational onto its image (since $4$ and $3$ are relatively prime), contradicting $g(X)=7$. Therefore $\bar{X}$ is bielliptic, so there is a degree \(2\) covering $\phi: \bar{X} \to C$, for an elliptic curve \(C\). Consider the composite map $\bar{X} \to C \times Y_2$. Since $X$ has genus $7$, it cannot be birational to a $(2,4)$ curve (genus would be at most \(5\)) on $C \times Y_2$, therefore the morphism $\bar{X} \to Y_2$ factors through $\phi$. Similarly, the morphism $\bar{X} \to Y_1$ factors through $\phi$, contradicting the birationality of $X \to Z$.
\end{proof}
Motivated by Theorem \ref{2D-is-birational} we make the following definition.

\begin{definition}
Suppose $X/k$ is a curve with $\min(\dendegs(X/k)) = d$. We say that $X$ is \defi{$d$-minimal} if there does not exist a covering of curves $\pi: X \to Y$ of degree at least $2$ such that $\min(\dendegs(Y/k)) \deg \pi = d.$ 
\end{definition}

The problem of understanding the minimum density degree is reduced, by Theorem \ref{2D-is-birational}, to analyzing the geometry of $d$-minimal curves. Note that Theorem \ref{2D-is-birational} already gives us some control over this geometry: since a $d$-minimal curve $X$ has a birational embedding of degree $2d$, the genus of $X$ is bounded by $(d-1)(2d-1)$.  We will prove a stronger genus bound in Theorem \ref{thm:nD_lower_bound} below. 

In our analysis of $d$-minimal curves $X$ we occasionally need to use hyperbolicity of $X$; the following lemma allows us to do so.
\begin{lemma}\label{lem:genus-at-least-2}
Suppose $X$ is a $d$-minimal curve with $d \geqslant 2$. Then the genus of $X$ is at least $3$. 
\end{lemma}
\begin{proof}
If $X$ has genus zero, then it is isomorphic to a plane conic. If $X(k)=\emptyset$, then projection from a rational point on the plane defines a degree $2$ map $X \to \PP^1$ and $\min(\dendegs(X/k)) =2$. If $X(k) \neq \emptyset$, then $X = \PP^1$ and $X$ is $1$-minimal.

If $X$ has genus $1$ and infinitely many rational points, then $X$ is $1$-minimal. Otherwise, by Riemann--Roch, if $D$ is a rational degree $d\geqslant 2$ divisor on $X$, then $\dim H^0(X, \calO(D)) \geqslant 2$, and so $X$ is a degree $d$ cover of $\PP^1$.

If \(X\) has genus \(2\), then \(\min(\dendegs(X/k)) \geq 2\) by Faltings' Theorem.  On the other hand, the canonical linear series exhibits that \(X\) is a degree \(2\) cover of \(\pp^1_k\), for which \(\min(\dendegs(\pp^1/k)) = 1\).
\end{proof}

The birationality of $|2D|$ on a $d$-minimal curve implies a classification of $2$-minimal curves. The resulting theorem is an arithmetic strengthening of \cite[Corollary 3]{Harris-Silverman1991}, which proves that any curve with \(\min(\dendegs(C/k)) = 2\) is a degree \(2\) cover of a genus \(0\) or \(1\) curve.

\begin{theorem}\label{thm:no_2_min}
There are no $2$-minimal curves over any number field.
\end{theorem}
\begin{proof}
Suppose to the contrary that we are in Setup \ref{main_setup} and $X$ is a $2$-minimal curve. By Theorem \ref{2D-is-birational}, for a general divisor $D \in A$ the linear system $|2D|$ is birationally very ample. A nondegenerate degree $4$ curve in $\pp^n$ for $n \geqslant 3$ has genus at most $1$. By Lemma \ref{lem:genus-at-least-2} we can assume that for a general $D$ the linear system $2D$ realizes $X$ as a plane quartic $Y_D \subset \PP^2$.

If $Y_D$ is smooth, then $X=Y_D$ is a canonical genus $3$ curve. In particular $2D=K_X$. Since $D$ was general, we can assume $2D \neq K_X$.
If \(Y_D\) is singular, its geometric genus (the genus of \(X\)) is at most \(2\), so \(X\) cannot be \(d\)-minimal for any \(d\geq 2\) by Lemma \ref{lem:genus-at-least-2}.
\end{proof}
\section{Subspace configurations}\label{sec:subspace-configurations}

We will analyze the geometry of $d$-minimal curves by studying structures (``subspace configurations'') associated to the birational linear systems $|nD|$, where $n \geqslant 2$ is an integer and $D$ is a degree $d$ point on $X$. We first establish notation for and basic properties of these objects, building to a proof of Theorem \ref{genus-bound}. We summarize the discrete-geometric structure of these subspace configurations in Section \ref{sec:construction-summary}.

From now on we use notation of Setup \ref{main_setup} and assume additionally that $X$ is $d$-minimal.

\subsection{Geometric considerations}

Given an abelian translate $A \hookrightarrow W_dX$, the tensor product map on line bundles gives a map
\[\underbrace{A \times A \times \cdots \times A}_{n} \to W_{nd}X, \]
whose image $A^{(n)}$ is (noncanonically) isomorphic to $A$ (since we assumed $A$ is a trivial torsor).  Every divisor in $A^{(n)}$ is (geometrically) of the form $nD$ for some $D \in A$.  By Theorem \ref{2D-is-birational} the linear system $|nD|$, for $n \geqslant 2$, is birationally very ample.  By upper-semicontinuity of dimensions of global sections, there is an open subset of $D$ in $A$ with the same (minimal) value of $\dim |nD|$; we denote this minimal value by $r(n)$ (so in fact $A^{(n)} \subset W^{r(n)}_{nd}X$.)

Given any divisor $D'$ on $X$, there is an evaluation map
\[ H^0(X, nD) \xrightarrow{\ev_{D'}} \O(nD)_{D'}, \]
whose kernel is identified with the space of sections vanishing along $D'$.  If we let $D'$ vary among the divisors parameterized by $A$, the dimension of kernel is an upper-semicontinuous and achieves a generic value on an open subset of $A$.  We write $s(n)$ for $r(n)$ minus this generic dimension of $h^0(X, nD - D')$ as $D'$ varies over the divisors parameterized by $A$.

The number $s(n)$ has a geometric interpretation in terms of the map to projective space given by $|nD|$.  We will write $\Span_{|nD|}(D')$ for the linear span of the images of the points of $D'$ under the map $|nD|$.  Then $\Span_{|nD|}(D')$ is a projective space of dimension at most $s(n)$.  For $D' \in A$ general, $\Span_{|nD|}(D')$ has dimension exactly $s(n)$.  (When the linear system $nD$ is unambiguous, we will implicitly write $\Span(D')$.) The abelian translate property from Section \ref{sec:abelian-translate} in this geometric language says that for any collection of $n-1$ divisors $D_1, \dots, D_{n-1}$ from $A$, there exists a divisor $D_n$ such that their spans $\Span_{|nD|}(D_1), \dots, \Span_{|nD|}(D_{n})$ in $|nD|$ are contained in a common hyperplane.

\begin{lemma}\label{lem:spans_meets_divisor}  Let \(X\) be \(d\)-minimal.
Suppose that \(D_1\) and \(D_2\) are general divisors from \(A\) and that \(D\) is an independently general divisor from \(A\).  
\begin{itemize}
\item If \(n \geqslant 3\), then \(X \cap \Span_{|nD|}D_1 = D_1\).
\item If \(n = 2\), then \(X \cap \Span_{|2D|} D_1  = D_1 \sqcup (2D - D_1)\).
\end{itemize}
In particular, \(D_2 \cap \Span_{|nD|} D_1 = D_1 \cap D_2\).
\end{lemma}
\begin{proof}
First suppose that $n \geqslant 3$. Since \(D\) is general, the line bundle \(nD - D_1\) is basepoint free by Proposition \ref{prop:elliptic-exceptions}.  On the other hand, any point of  \((\Span_{|nD|} D_1)\cap D_2\) that is not a point of \(D_1\) would be a basepoint of \(nD - D_1\).

Now suppose $n=2$. Since $D_1, D$ are a general pair, the space $\Span_{|2D|} D_1$ is a hyperplane, for otherwise the projection from  a codimension $2$ space containing $\Span_{|2D|} D_1$ is a degree $d$ (or less) map from $X$ to $\PP^1$. Since $D_1+(2D-D_1)=2D$, the hyperplane section $X \cap \Span_{|2D|} D_1 $ equals $D_1 \sqcup (2D-D_1).$ 
\end{proof}

By definition, $r(n)-r(n-1)=s(n) + 1$. The difference $s(n)-s(n-1)$ also has a geometric interpretation.

\begin{lemma}\label{lem:s_n_minus_one}
We have \(s(n) - s(n-1) =  \lambda +1\), where \(\lambda= \dim \left(\Span_{|nD|}(D_1) \cap \Span_{|nD|}(D_2)\right)\), for general \(D_1, D_2 \in A\).
\end{lemma}
\begin{proof}
Since \(D_1\) and \(D_2\) are general, we have \(D_1 \cap D_2 = \emptyset\) by \eqref{eq:disjoint_support}, and so Lemma \ref{lem:spans_meets_divisor} guarantees that the projection of $\Span_{|nD|} D_2$ from $\Span_{|nD|} D_1$ is $\Span_{|nD-D_1|} D_2$. Since $D, D_1, D_2$ are general, we have $s(n)=\dim\Span_{|nD|} D_2$ and $s(n-1)=\dim\Span_{|nD-D_1|} D_2$. 
\end{proof}

We want to analyze the geometry of the configuration of $\Span D'$ in $|nD|$ for various $n\geqslant 2$. It will be convenient to project from a maximal subspace that is common to $\Span D'$ for almost all $D'$; to formalize this we make the following definition. 

\begin{definition}[Definition/Notation]
For a dense open subset $W \subset A$, let 
\[V_{|nD|, W}\colonequals \bigcap_{D' \in W} \Span_{|nD|}D'.\] 
Let $V_{|nD|}$ be the maximal subspace of the form $V_{|nD|,W}$ as $W$ varies over dense opens in $A$.  Explicitly, \(V_{|nD|} = V_{|nD|, W}\) for \(W\) the locus of \(D'\) where \(\Span D'\) has the maximal dimension \(s(n)\).
\end{definition}

\begin{lemma}\label{elliptic-case}
Suppose $X$ is a $d$-minimal curve and $D \in A(k)$ is a general divisor. Then the codimension of $V = V_{|2D|}$ in $|2D|$ is at least $3$.
\end{lemma}
\begin{proof}
Suppose that to the contrary the codimension of $V$ is equal to $2$. The projection from $V$ defines a morphism $\pi_V: X \to \PP^1$ of degree at most $2d$. Since for a general divisor $D'$, $\Span_{|2D|}D'$ is contained in a hyperplane and contains $V$, a general divisor $D' \in A$ is contracted to a point by $\pi_V$. In particular the divisors from $A$ vary in a one-dimensional family, and so $\dim A = 1$. Moreover, since $D'$ and $2D-D'$ belong to the same hyperplane, $\pi_V(D')=\pi_V(2D-D')$. For general $D'\in A$, the divisors $D'$ and $2D-D'$ don't share points by \eqref{eq:disjoint_support}, and so the degree of $\pi_V$ equals $2d$.  By Lemma \ref{lem:ramification_argument} this implies that \(X\) is a degree \(d\) cover of the elliptic curve \(A\), which contradicts our assumption of \(d\)-minimality. 
\end{proof}

In Lemma \ref{two-divisors-through-a-point} we showed that when \(X\) is \(d\)-minimal, a general point $P$ on $X$ is contained in at least two distinct divisors from $A$. It will be convenient for us to consider separately the case when a general point $P$ in $X$ is the intersection of exactly two divisors from $A$. We refer to this property as condition \eqref{dagger}, formalized as follows:
\begin{equation}\tag{$\dagger$}\label{dagger}
\text{For a general point }P \in X\text{ there exists  a pair of divisors } F, F' \in A \text{ such that } F \cap F'=P
\end{equation}

We do not know of examples in which condition \eqref{dagger} fails for a $d$-minimal curve $X$. We have the following sufficient condition for \eqref{dagger}:

\begin{lemma}\label{lemma:two_divisors_one_point}
Suppose $X$ is $d$-minimal and $r(2)=2$. Suppose $D\in A$ is general, $x \in D$ is a point, and $D'$ is a general divisor containing $x$. Then $D \cap D' = \{x\}$. In particular, condition \eqref{dagger} holds.
\end{lemma}
\begin{proof}
Choose a general divisor $E$ disjoint from $D$ and $D'$. The pair \((x, D + E)\) is a general point of \(X \times A^{(2)}\).  In particular, \(x\) is a smooth point on the image of \(X \subset |D+E|\simeq \pp^2\).  The span of $D$ in $|D+E|$ is a line $\ell$ that intersects the curve in $D \cup E$. The span of $D'$ is a line $\ell'$ that does not equal $\ell$ since $D' \not\subset D \cup E$. A pair of distinct lines shares exactly one point, and so $D \cap D' \subset \ell \cap \ell' = \{x\}.$
\end{proof}
Under condition \eqref{dagger} the linear configuration of $\Span D'$ in $|nD|$ has interesting incidence structure, as we show in Proposition \ref{extra-intersections}. We first need to prove the following linear nondegeneracy property of $D \subset \Span D$. 
\begin{lemma}\label{transitive-monodromy}
Suppose $n \geqslant 2$ is an integer such that $s(n) \leqslant d-2$. Let $D \in A$ be a general divisor. Then for a general divisor $D' \in A$ and any point $x \in D'$ we have $\Span_{|nD|} (D'\setminus\{x\})=\Span_{|nD|} D'$. 
\end{lemma}
\begin{proof}
Suppose that for a general divisor $D'$ there is a point $x \in D'$ such that the set $D' \setminus \{x\}$ is contained in a hyperplane inside $\Span D'$. Choose a divisor $D' \in U(k)$ such that the Galois group $G_k$ acts transitively on $D'$ and the complement of a point $x \in D'$ belongs to a hyperplane $H \subset \Span D'$. By transitivity of the Galois action, for every $x \in D'$ there exists a hyperplane $H_x$ that contains $D' \setminus \{x\}$. Choose points $x_1, ..., x_{s(n)+1} \in D'$ that span $\Span D'$.  Since $s(n)+1 \leqslant d-1$ there is a point $x \in D'$ such that $x \neq x_i$. Then $H_x$ would contain the points $x_1, ..., x_{s(n)+1}$, so $H_x$ contains their span $\Span D'$. This is a contradiction.  
\end{proof}

\begin{proposition}\label{extra-intersections}
Suppose $X$ is a $d$-minimal curve, condition \eqref{dagger} holds, and $n \geqslant 2$ is an integer. 
Suppose that for a general $D' \in A$,  we have $\dim \Span_{|nD|}D' \leqslant d-2$. Then for a general pair of divisors $D_1, D_2 \in A$ we have $\Span_{|nD|} D_1 \cap \Span_{|nD|} D_2 \neq V_{|nD|}.$
\end{proposition}
\begin{proof}
    If $n=2$, $D' \in A$ is general, and $\Span_{|2D|} D'$ is not a hyperplane, then projection from $\Span_{|2D|} D'$ defines a morphism from $X$ to a positive-dimensional projective space of degree at most $d$. Therefore $\Span_{|2D|} D'$ is a hyperplane, and so for a general pair $D_1, D_2$, we have that $\Span_{|2D|} D_1 \cap \Span_{|2D|} D_2$ has codimension $2$.  Since \(V_{|2D|}\) has codimension at least \(3\) by Lemma \ref{elliptic-case}, the conclusion holds.
        
 Assume for the remainder of the proof that $n \geqslant 3$.
    Consider the linear system $(n-1)E$ for a general $E \in A$ and another general divisor $F \in A$. Choose a point $x \in F$ and a divisor $F'$ such that $F \cap F'=\{x\}$; this is possible since \eqref{dagger} holds. Since $F$ was general, $F'$ is general as well (although the pair \(F, F'\) is not general). In particular $\dim\Span_{|(n-1)E|}F'=s(n-1)$. Consider the linear system $|(n-1)E+F|$. Since $E$ and $F$ are general, $|(n-1)E+F|=|nD|$ for a general $D$ and $F, D$ form a general pair.  The points of $F$ do not belong to $V_{|nD|}$ (for example, by Lemma \ref{lem:spans_meets_divisor}). Therefore, both $\Span_{|nD|}F$ and $\Span_{|nD|}F'$ contain the point $x$, which is outside of $V_{|nD|}$, so
    \[\Span_{|nD|}F \cap \Span_{|nD|}F' \neq V_{|nD|}.\]
    We have $F \cap F'=x$ by construction.  Considering the projection \(\pi\) from \(\Span_{|nD|} F\), we have
    \begin{equation}\label{dimproj1}\dim \pi(\Span_{|nD|} F') = \dim \Span_{|nD|} F' - \dim(\Span_{|nD|}F \cap \Span_{|nD|}F') - 1 < s(n) - \dim V_{|nD|} - 1. \end{equation}
    On the other hand, 
 \begin{align}\label{dimproj2}
 \pi(\Span_{|nD|} F') &=  \pi(\Span_{|nD|} (F' \smallsetminus \{x\})) & (\text{Lemma  \ref{transitive-monodromy}})\\
 &=  \Span_{|nD-F|} \left( F' \smallsetminus \{x\}\right) &  \text{(Lemma \ref{lem:spans_meets_divisor})} \notag \\
 &= \Span_{|(n-1)E|} \left(F' \smallsetminus \{x\} \right) \notag \\
 &= \Span_{|(n-1)E|} F'  & (\text{Lemma  \ref{transitive-monodromy}}).\notag
 \end{align}
    Combining \eqref{dimproj1} and \eqref{dimproj2} we see
    \[s(n-1) = \dim \Span_{|(n-1)E|} F' <  s(n) - \dim V_{|nD|} - 1.\]
    Therefore, by Lemma \ref{lem:s_n_minus_one}, the intersection of a general pair of divisor spans is larger than \(V_{|nD|}\).
\end{proof}

We are now in the position to analyze the geometry of linear systems obtained by projecting \(|nD|\) from the subspace \(V_{|nD|}\).
To do so we introduce the following definition.

\begin{definition}
Suppose $n\geqslant 2$ is an integer. We denote by $|nD|'$ the linear system obtained from $|nD|$ by projection from $V_{|nD|}$. Similarly, let $r'(n)=\dim |nD|'$ and $s'(n)=\dim_{|nD|'}\Span D'$ for general $D, D' \in A$.
\end{definition}

\noindent Proposition \ref{extra-intersections} immediately implies:
\begin{corollary}\label{cor:s(n)-grows}
Suppose \eqref{dagger} holds. Then we have $s'(n) \geqslant \min(s(n-1)+1, d-1)$.
\end{corollary}
\begin{proof}
Suppose $s'(n) \leqslant d-2$. By Proposition \ref{extra-intersections}, for general $D_1, D_2 \in A$ the spaces \(\Span_{|nD|'} D_1\), \( \Span_{|nD|'} D_2\) share a point. Since \(D_1, D_2\) are general, no point of \(D_2\) belongs to \(\Span_{|nD|'} D_1\) by Lemma \ref{lem:spans_meets_divisor}.  Projecting from \(\Span_{|nD|'} D_1\) we get 
\[\dim \Span_{|nD|'} D_2 \geqslant \dim \Span_{|nD - D_1|} D_2+1.\]  Therefore  \(s'(n) \geqslant s(n-1)+1.\)
\end{proof}

\begin{theorem}\label{thm:nD_lower_bound}
Suppose that we are in Setup \ref{main_setup} and $X$ is $d$-minimal. Suppose \eqref{dagger} holds. Then for a general divisor $D \in A(k)$ and every number $n \leqslant d$ we have 
\[\dim |nD| \geqslant \dim |nD|' \geqslant \frac{n(n+1)}{2} - 1.\]
\end{theorem}
\begin{proof}
By Lemma \ref{elliptic-case}, we have $r'(2) \geqslant 2$ and $s'(2)\geqslant 1$. Combining \(s'(2) \geqslant 1\) with Corollary \ref{cor:s(n)-grows}, we have \(s(n) \geqslant s'(n) \geqslant \min(d-1, n-1)\). Therefore for $2<n\leqslant d$,  
\begin{align*}
r'(n) &= (s'(n) + 1) + r(n-1) = (s'(n)+1) + (s(n-1) + 1) + \cdots + (s(3) + 1) + r(2) \\
&\geqslant \frac{n(n+1)}{2} - 1.\qedhere
\end{align*}
\end{proof} 
Theorem \ref{thm:nD_lower_bound} can be used to bound the genus of curves for which  condition \eqref{dagger} holds. When \eqref{dagger} does not hold, we use the following lemma instead.

\begin{lemma}\label{genus-fix-1}
Suppose we are in Setup \ref{main_setup} and $X$ is $d$-minimal. Suppose $r(2)\geqslant 3$ and $d \geqslant 4$. Then $r'(3) \geqslant 7$.
\end{lemma}
\begin{proof}
Because \(X\) does not admit a degree \(d\) map to \(\pp^1\), we have \(s(2)=r(2)-1 \geqslant 2\).  We also have $s'(3) \geqslant s(2)$, and $r'(3)=s'(3)+r(2)+1$. Therefore the only case in which $r'(3)=6$ is $r(2)=3$, $s'(3)=s(2)=2$.

By Lemma \ref{two-divisors-through-a-point}, for a general point $P \in X$ there exists a pair of divisors $D_1, D_2 \in U$, $D_1 \neq D_2$ both containing $P$. We can assume that both $D_1, D_2$ are general divisors. If condition \eqref{dagger} holds, then $s'(3)> s(2)$ by Corollary \ref{cor:s(n)-grows}, contradicting our calculation that \(s'(3)=s(2) =2\) above.  We may therefore assume that $D_1 \cap D_2$ contains at least $2$ points, but that there exists some point \(y \in D_2 \smallsetminus D_1\). Choose a general divisor $D \in U$. By Lemma \ref{lem:spans_meets_divisor}, the point \(y\) is in \(\Span_{|2D|}D_2\) but not in $\Span_{|2D|}D_1$. Therefore the planes $\Span_{|2D|} D_1$ and $\Span_{|2D|}D_2$ intersect along a line.

 By Lemma \ref{transitive-monodromy} the intersection $D_1 \cap D_2$ contains at most $d-2$ points.  Consider now the embedding of $X$ into $\PP^6$ given by the linear system $|2D+D_2|'$. Since $D$ was general, the spans of $D_1$ and $D_2$ in this linear system have dimension $s'(3)=2$. However $\Span_{|2D+D_2|'}D_1$ and $\Span_{|2D+D_2|'} D_2$ share at least $2$ points, and therefore intersect along a line, again applying Lemma \ref{lem:spans_meets_divisor}.  Therefore the projection from $\Span_{|2D+D_2|'}D_2$ maps all points of $D_1 \smallsetminus D_2$ onto a single point in $|2D|$. Since $D_1 \smallsetminus D_2$ contains at least $2$ points, this is a contradiction.
\end{proof}

A similar argument can be used to improve the estimate for the value of $r(4)$; it will be useful in our considerations of low values of $d$.

\begin{lemma}\label{genus-fix-2}
Suppose we are in Setup \ref{main_setup} and $X$ is $d$-minimal, and $d \geqslant 5$ is odd. Suppose \eqref{dagger} does not hold. Then $r'(4) \geqslant 12$
\end{lemma}
\begin{proof}
Since \eqref{dagger} does not hold, we have $r(2)\geqslant 3$ by Lemma \ref{lemma:two_divisors_one_point}.
Because \(X\) does not admit a degree \(d\) map to \(\pp^1\), we have \(s(2) = r(2)-1\).  Furthermore, by considering projections from divisor spans, we see that \(s'(n) \geqslant s(2)\) for all $n>2$, and that \(r'(4) \geqslant r(2) + 2s(2)+2\).
Combining these, if $r(2)\geqslant 4$, then $r'(4) \geqslant 3r(2) \geqslant 12$. It therefore suffices to consider the case $r(2)=3$.  In this case $r'(3)\geqslant 7$ by Lemma \ref{genus-fix-1}, and so $s(4) \geqslant s(3) \geqslant 3$. 
Again considering projection, we have
\[r'(4) = s'(4) + 1 + r(3) \geqslant s(3) + 1 + r'(3) \geqslant 3 + 1 + 7 =11.\]
Therefore the only way to have $r'(4) < 12$ is to have equality everywhere, and hence $r(2)=3$, $r(3) = r'(3)=7$, $s(3)=s'(4)=3$, and $r'(4)=11$.

Consider a general pair of divisors $D, D_1 \in A$. Suppose $D_2\in A$ is a divisor that shares points with $D_1$.  Note that $D_2$ is a general divisor in $A$ since $D_1$ is general, and moreover, since \eqref{dagger} does not hold, we can assume that $D_1 \cap D_2$ contains at least $2$ points.

Consider the linear system $|3D| = |3D|'$ and the subspaces $\Span_{|3D|}D_1$ and $\Span_{|3D|} D_2$. By assumption these are both $3$-dimensional. Since $\Span_{|3D|}D_1$ and $\Span_{|3D|} D_2$ have nontrivial intersection but do not coincide by Lemma \ref{lem:spans_meets_divisor}, Lemma \ref{transitive-monodromy} implies that $D_1 \setminus D_2$ contains at least $2$ points. Suppose $\dim \Span_{|3D|}D_1 \cap \Span_{|3D|}D_2 = 2$. Projecting from $\Span_{|3D|}D_2$ we see that in the linear system $|3D-D_2|$, the points of $D_1 \setminus D_2$ map to a single point. Since $D_1$ and $|3D-D_2|$ is a general pair of divisors, this is a contradiction. Therefore $\dim \Span_{|3D|}D_1 \cap \Span_{|3D|}D_2 = 1$, and so all points of $D_1 \cap D_2$ in $|3D|$ belong to a single line.

Consider the linear system $|3D + D_2|'$.  Since \(D\) is general, \(3D + D_2\) is a general point of \(A^{(4)}\). By assumption the subspaces $\Span_{|3D + D_2|'}D_1$ and $\Span_{|3D + D_2|'} D_2$ are $3$-dimensional, distinct, and meet in at least \(2\) points. Therefore the projection of $D_1 \setminus D_2$ from $\Span_{|3D + D_2|'} D_2$ maps to a space of dimension at most $1$ in \(|3D|\). Thus for a general divisor $D$, the image of $D_1$ in $|3D|$ is contained in a pair of skew lines each containing at least $2$ points (since $D_1 \setminus D_2$ and $D_1 \cap D_2$ both contain at least $2$ points). A nondegenerate set $S$ of $d \geqslant 5$ distinct points in \(\pp^3\) is contained in at most one pair of skew lines with each line containing at least $2$ points. Therefore the pair of lines $\Span_{|3D|} (D_1 \setminus D_2)$ and $\Span_{|3D|} (D_1 \cap D_2)$ are preserved by the Galois action on $D_1$, and, in particular, each line has to contain the same number of points of $D_1$. This contradicts the assumption that $d$ is odd.
\end{proof}

We now prove Theorem \ref{genus-bound} from the introduction.

\begin{theorem}\label{thm:weak-genus-bound}
Given an integer $d$, let $m \colonequals \lceil d/2\rceil -1$ and let $\varepsilon \colonequals 3d-1-6m<6$. Suppose $X$ is a $d$-minimal curve. If \eqref{dagger} holds, then the genus of $X$ is bounded by \[d(d-1)/2 + 1.\] If \eqref{dagger} does not hold, then the genus is at most \[ 3m(m-1)+m\varepsilon.\]
\end{theorem}
\begin{proof}
If \eqref{dagger} holds, this follows from Theorem \ref{thm:nD_lower_bound} for $n=d$ and Castelnuovo's genus bound (see \cite[Chapter III, page 116]{ACGH} for the proof of the bound, and Section \ref{sec:Applications} Equation \eqref{Castelnuovo-function} for the statement).  Alternatively, since any nondegenerate special curve in \(\pp^r\) has degree at least \(2r\), the linear system \(|dD|\) on \(X\) is nonspecial for \(d > 2\) and the genus bound follows from Riemann--Roch.

If \eqref{dagger} does not hold, this is Castelnuovo's bound for a degree $3d$ curve in $\PP^7$, which applies by Lemma \ref{genus-fix-1}.
\end{proof}

An immediate corollary of this bound is the theorem of Abramovich--Harris on degree $3$ points on curves.
\begin{corollary}\label{cor:3-min}
Suppose $\min(\dendegs(X/k)) =3$. Then $\bar{X}$ is a degree $3$ cover of $\pp^1$ or an elliptic curve.
\end{corollary}
\begin{proof}
If $X$ is not $3$-minimal the conclusion holds, so we can assume $X$ is $3$-minimal. Then by Theorem \ref{thm:weak-genus-bound} the genus of $X$ is at most $4$. The geometric gonality of a curve of genus $g \leqslant 4$ is at most $3$.
\end{proof}

\subsection{Summary of Setup and Notation}\label{sec:construction-summary}

We give a brief summary of the basic structures and properties introduced in the previous section. We fix a $d$-minimal curve $X$, and let $A \subset W_d X$ be a corresponding abelian variety with dense \(k\)-points. For every $D \in A$ and every integer $n \geqslant 2$ we consider the linear system $|nD|$ and the corresponding projective embedding of $X$. Within the projective space $\PP^{|nD|}$, we look at the linear spaces of the form $\Span_{|nD|} E$ for all divisors $E \in A$. The resulting system of subspace configurations enjoys a number of unusual properties. We use $V=V_{|nD|}$ to denote the maximal subspace shared by all spaces $\Span_{|nD|} E$ for a Zariski open family of $E \in A$. Projecting from $V$ defines the linear system $|nD|'$ on $X$ equipped with a similar family of linear spaces $\Span_{|nD|'} E$. The basic properties of these structures are the following:

\begin{enumerate}
    \item The dimensions of $|nD|$ and $|nD|'$ have fixed values $r(n), r'(n)$ for a generic choice of $D \in A$;
    \item For general $D, E$ the dimensions of $\Span_{|nD|}E$ and $\Span_{|nD'|} E$ have constant values $s(n), s'(n)$;
    \item If \eqref{dagger} holds, and $n$ is such that $s(n) \leqslant d-2$, then for a general $D$ and a general pair $E_1, E_2$ the subspaces $\Span_{|nD|'}E_1$ and $\Span_{|nD|'}E_2$ have nonempty intersection; 
    \item The intersection of all subspaces $\Span_{|nD|'} E$ as $E$ varies over any Zariski open subset in $A$ is empty. 
    \item\label{atp} (The \defi{abelian translate property}) For any $D\in A$ and any divisors $E_1, ..., E_{n-1} \in A$ there exists a divisor $E_n \in A$ such that the subspaces $\Span_{|nD|'}E_i$ all belong to the same hyperplane;
    \item For general $D, E \in A$ the projection of $X$ in $|nD|$ from $\Span_{|nD|}E$ is equivalent to the embedding given by $|nD - E|$; in particular we have the identities $r(n)-s(n)=r'(n)-s'(n)=r(n-1) + 1.$
\end{enumerate}
The presence of these linear configurations allows us to give various restrictions on the geometry of the curve $X$. In the next section we will use this structure to identify the curves with $r(2)=2$ with the curves constructed by Debarre and Fahlaoui \cite{Debarre-Fahlaoui1993}. 

\section{Debarre--Fahlaoui curves}\label{sec:Debarre-Fahlaoui}

Let $A$ be a positive rank elliptic curve over $k$.
For all $d \geqslant 4$, Debarre and Fahlaoui give examples of $d$-minimal curves lying on the smooth surface $\Sym^2 A$.  We first recall their construction, and then we show that any \(d\)-minimal curve with \(r(2) = 2\) naturally arises in this way. This shows that the simplest class of $d$-minimal curves is the one provided by the Debarre--Fahlaoui construction.

We begin by recalling the setup from \cite[Section 4.1]{Debarre-Fahlaoui1993}.
The addition law on $A$ induces a natural map $\pi \colon \Sym^2 A \to A$.
Let $o \in A(k)$ be the origin, and let $\EE$ be the unique nonsplit extension
\[ 0 \to \O_A \to \EE \to \O_A(o) \to 0.\]
Then the fibration $\Sym^2 A \xrightarrow{\pi} A$ is isomorphic to $\pp \EE \to A$.  Let $H$ denote the relative $\O_{\pp \EE}(1)$.   Then we have 
\[\Pic(\Sym^2 A) \simeq \pi^*\Pic(A) \oplus \mathbb{Z} H.\]
We will write $F_x$ for the divisor $\pi^*\O_A(x)$; in terms of the moduli description of $\Sym^2 A$, this consists of all degree $2$ effective divisors on $A$ that sum to $x$ under the group law.
The divisors $F_x$ for all $x \in A$ are numerically equivalent, and we simply write $F$ for this numerical class. 
Another natural divisor on $\Sym^2 A$, which we denote $H_x$ consists of all effective divisors of degree $2$ on $A$ that contain $x$.  The rational equivalence class of this divisor is
\(H_x = H - F_o + F_x\) \cite[Section 4.1 (ii)]{Debarre-Fahlaoui1993}.
(In particular, $H_o = H$.)
The numerical classes of divisors are spanned by $H$ and $F$, with the following intersection relations:
\[H^2 = 1, \qquad H \cdot F = 1, \qquad F^2 = 0.\]
The canonical class $K$ on $\Sym^2 A$ has numerical class
\(K = -2H + F\).
The Nef and effective cones both consist of all classes $aH + bF$ where $a \geqslant 0$ and $a + 2b \geqslant 0$ \cite[Chapter V, Proposition 2.21]{Hartshorne}.

\begin{definition}
A \defi{Debarre--Fahlaoui curve} is a geometrically integral curve on $\Sym^2 A$ in numerical class
\[(d+m)H - mF,\]
for some $1 \leqslant m \leqslant d$.
\end{definition}

This terminology comes from the fact that Debarre and Fahlaoui consider the case $m=1$ in \cite{Debarre-Fahlaoui1993} to give counterexamples to the conjecture of Abramovich--Harris \cite[page 229]{Abramovich-Harris1991}.  Let us recall this construction.

Let $X$ be a Debarre--Fahlaoui curve.  The family of divisors $H_x$ on $\Sym^2 A$ restricts to a family of degree $d$ effective divisors on $X$ parameterized by $A$, since $H \cdot ((d+m)H - mF) = d$.  This family gives rise to an embedding
\[\psi \colon A \hookrightarrow W_dX.\]
This family of degree $d$ divisors is \emph{not} induced by a map $X \to A$: if $H_x \cdot X$ contains the degree $2$ effective divisor $[x + x']$, then so does $H_{x'} \cdot X$; as such these cannot be (the necessarily disjoint) fibers of a map.

\begin{proposition}[{\cite[Propositions 5.7 and 5.14]{Debarre-Fahlaoui1993}}]\label{prop:DFs-are-minimalish}
Let $d \geqslant 4$ and $1 \leqslant m \leqslant d$ be integers.
Consider the numerical class $(d+m)H - m F$.
\begin{enumerate}
    \item If $m < d/2$, then for any nice curve $X$ in this class, we have $\gon \bar{X} > d$.
    \item If the class of $X$ is very ample (e.g., if \(m=1\)), then a general curve in this class admits no nontrivial maps of degree at most $d$ to a non-isomorphic curve of genus at least \(1\). 
\end{enumerate}
In particular, under both of these assumptions, such a curve $X$ is $d$-minimal.
\end{proposition}

\medskip

We now turn to \(d\)-minimal curves with \(r(2) = 2\).
We will analyze the geometry of such curves by looking at the induced subspace configurations in $|3D|'$.  We begin by showing, in Lemma \ref{lem:rprime_3}, that this structure is a configuration of $2$-planes in a $5$-space; in other words \(r'(3) = 5\) (which forces \(s'(3)=2\)). 

\begin{lemma}\label{lem:linear-algebra}
Let $V_i \subset \PP^n, i \in I$ be a collection of codimension \(2\) subspaces of $\PP^n$ spanning all of \(\pp^n\). Suppose that for any $i, j \in I$, the subspaces $V_i, V_j$ belong to a common hyperplane. Then there is a codimension $3$ subspace $\Lambda$ that belongs to $V_i$ for all $i \in I$. 
\end{lemma}
\begin{proof}
Since $V_i$ have codimension $2$, if $V_i, V_j$ are two distinct subspaces, then $\dim \Span (V_i, V_j)=n-1$, and so $\dim V_i \cap V_j = n-3$. Choose a subspace $V_k$ such that $V_k$ does not belong to $\Span (V_i, V_j)$. Then $\dim V_k \cap \Span (V_i, V_j) = n-3$, and on the other hand $V_k$ intersects each of $V_i$ and $V_j$ in a subspace of dimension $n-3$. Therefore $V_k$ contains $V_i \cap V_j$. Finally take any subspace $V_\ell$, $\ell \neq i,j,k$. Then $V_\ell$ does not belong to $\Span (V_u, V_w)$ for some $u,w \in \{i,j,k\}$ and by the previous argument $V_\ell$ contains $V_u \cap V_w = V_i \cap V_j$. Thus $\Lambda=V_i \cap V_j$ satisfies the conclusion of the lemma.
\end{proof}

\begin{lemma}\label{lem:no-extra-intersecions}
Let \(X\) be a \(d\)-minimal curve with \(r(2) = 2\).
Suppose that $P \in X$ is a general point, and $D \in A$ is general. Then there exists a pair of distinct divisors $D_1, D_2$  from $A$ through $P$ such that $\Span_{|3D|'}D_1$ and $\Span_{|3D|'}D_2$ span a hyperplane in $|3D|'$.
\end{lemma}
\begin{proof}
By Lemma \ref{two-divisors-through-a-point} there exists a pair of distinct divisors $D_1, D_2$ through a general point of $X$, and moreover each $D_i$ individually is general in $A$. Choose such a pair.
Recall that since \(r(2)=2\) and \(r'(3) - r(2) = s'(3) + 1\), the spaces $\Span_{|3D|'}D_i$ have codimension \(3\) in \(|3D|'\).  By Lemma \ref{lemma:two_divisors_one_point}, we may assume that \(D_1\) and \(D_2\) meet only at \(P\). 

Consider the divisor $E \in A$ such that $3D-D_1=2E$; such $E$ is a general point of $A$, and in particular $|2E|$ is an embedding along each $D_i$. Note that $\Span_{|3D|'}(D_1+D_2)$ has codimension either \(1\) or \(2\).
If $\Span_{|3D|'}(D_1+D_2)$ has codimension \(2\), then the projection $\pi$ from $\Span_{|3D|'}D_1$ sends the points of $D_2 \smallsetminus D_1$ to the same point in $|2E|$. The set $D_2 \smallsetminus D_1$ contains at least $d-1$ points, contradicting the assumption that $|2E|$ is an embedding along \(D_2\).
\end{proof}

\begin{lemma}\label{lem:rprime_3}
Suppose $r(2)=2$. Then $r'(3)=5$.
\end{lemma}
\begin{proof}
We will consider the configuration of divisor spans in \(|3D|'\).
Since \(r(2) = 2\), we have $s'(3)=r'(3)-r(2)-1=r'(3)-3$.  Since two general divisor spans span a hyperplane by Lemma \ref{lem:no-extra-intersecions}, their intersection has codimension \(5\).
Furthermore, $r'(3) \geqslant 5$ by Lemma \ref{lemma:two_divisors_one_point} and Theorem \ref{thm:nD_lower_bound}. Suppose $r'(3) \geqslant 6. $   

Let \(E\) be a general divisor in \(A\).  Consider a general collection of divisors \(D_1, \dots, D_N\) in \(A\).  The spaces $W_{i} \colonequals \Span D_i \cap \Span E$
form a collection of distinct codimension \(2\) subspaces of \(\Span E\), any two of which span a hyperplane \(\Span(D_i + D_j) \cap \Span E\).  For \(N\) large enough, the intersection of all \(W_{i}\) is empty (since we are working in \(|3D|'\)).  Thus by Lemma \ref{lem:linear-algebra}, there is a hyperplane $\Lambda$ in \(\Span E \) containing all \(W_i\).  Since $N$ was arbitrarily large, for a general $D' \in A$ we have $\Span D' \cap \Span E \subset \Lambda.$ Being contained in $\Lambda$ is a closed condition; therefore every divisor $D'$ for which the codimension of $\Span D'\cap \Span E$ in $\Span E$ is $2$ satisfies $\Span D' \cap \Span E \subset \Lambda$. This contradicts Lemma \ref{lem:no-extra-intersecions} 
applied to $D$, $D_1=E$ and a point $P$ of $E$ outside $\Lambda$.
\end{proof}

We will relate curves with $r(2)=2$ to Debarre--Fahlaoui curves as follows. If $r'(3)=5$, then the resulting configuration of divisor spans in \(\pp^5\) is a family of $2$-planes, parametrized by $A$, pairwise sharing points. This will naturally gives rise to a rational map $\psi\colon \Sym^2 A \to \PP^5$ sending a pair of divisors to the intersection of their spans. Since there are at least two divisors from $A$ through every point on $X$ we expect $X$ to be in $\psi(\Sym^2 A)$. In this way $X$ ``wants to be'' a curve on $\Sym^2 A$.

To realize this idea, we first need to reduce to the case $\dim A = 1$ (in Lemma \ref{lem:dimg1}) and establish a  nondegeneracy property of our configuration (Lemma \ref{lem:generically-not-in-hyperplane}).

\begin{lemma}\label{lem:dimg1}
Suppose that \(X\) is \(d\)-minimal and \(\dim A > 1\).  Then \(r(2) > 2\).
\end{lemma}
\begin{proof}
By Proposition \ref{prop:elliptic-exceptions} we have $r(2) \geqslant \dim A$. Thus it suffices to show that the case $\dim A = r(2)=2$ does not occur. We consider the cases $d=3$ and $d \geqslant 4$ separately.

Suppose $d \geqslant 4$. Choose a general divisor $D\in A$ and consider the rational map $\phi: A \to (\PP^2)^\vee$ that sends a divisor class $E$ to the line $\Span_{|2D|}E$. 

The set of $2d$ points on a general linear section of $X \subset \PP^2$ is thereby equipped with a nonempty collection of $d$-element subsets coming from $A$. But the monodromy of the linear section is the symmetric group (see, for example, \cite[Lemma, Chapter III, page 111]{ACGH}), and so every $d$-element subset of a general linear section of $X$ is a divisor from $A$. Therefore, for a general divisor $E \in A$, there exists another divisor $E' \in A$ such that both $E \cap E'$ and $E' \smallsetminus E$ consist of at least two points. Choose a general $E$ and consider the linear system $|2D+E|'$. By  Lemma \ref{lem:rprime_3}, we have $\dim |2D+E|' = 5$ and $\dim \Span_{|2D+E|'}E = \dim \Span_{|2D+E|'} E' = r'(3)-r(2)-1=2$. Since $E$ and $E'$ share at least two points, the $2$-planes $\Span_{|2D+E|'}E$ and $\Span_{|2D+E|'}E'$ share a line $\ell$. Therefore the projection of $\Span_{|2D+E|'}E'$ from $E$ is a single point. Therefore all the points of $E' \smallsetminus E$ are mapped to the same point under $2D$, which is a contradiction.

Suppose now that $d=3$. Consider the linear system $|3D|'$ and the associated embedding of $X$ in $\PP^5$. Consider a general point $P \in X$ and a general pair of divisors $D_P, D_P'$ through $P$. Such a pair does not share any points on $X$ except for $P$ by Lemma \ref{lemma:two_divisors_one_point}. By Lemma \ref{lem:no-extra-intersecions}  we have
\[\dim \Span_{|3D|'} D_P \cap \Span_{|3D|'} D_P' = 0.\] 

Choose a pair of general points $P,Q \in X$. Since the pair is general, projection from the line $\ell=PQ$ realizes $X$ as a degree $7$ curve in $\PP^3$. The projection from $\ell$ maps the divisor spans that contain $P$ or $Q$ to lines in $\PP^3$. Since $P,Q$ are a general pair of points, a general divisor $D_P$ containing $P$ and a general divisor $D_Q$ containing $Q$ form a general pair of divisors (as $P,Q$ vary), and so $\Span_{|3D|'} D_P$ and $\Span_{|3D|'} D_Q$ intersect at a point. By the above description of generic intersections, we can choose an infinite collection of divisors $D_{P}^{1}, D_P^2, \dots$ and $D_Q^{1}, D_Q^{2}, \dots$ such that the projections $\ell_i=\pi_\ell(\Span_{|3D|'}D_P^i)$ and $\ell_i'=\pi_{\ell}(\Span_{|3D|'}D_Q^i)$ form two families of lines with lines in each family pairwise skew, and $\ell_i \cap \ell_j' \neq \emptyset$ for all $i,j$. Such a pair of families is always contained in a smooth quadric.  Since every line $\ell_i, \ell_i'$ contains points from $\pi_\ell(X)$, the curve $\pi_\ell(X)$ shares infinitely many points with the quadric, and so belongs to the quadric. Therefore projection from $PQ$ realizes $X$ as a degree $7$, $(e_1, e_2)$-curve on the quadric. As $(P,Q)$ varies, the value of $(e_1, e_2)$ achieves a generic value on an open subset of $\Sym^2 X$; by monodromy, for this generic value $e_1=e_2$. However, $X$ has degree $7=e_1+e_2$, contradiction.
\end{proof}

\begin{lemma}\label{lem:generically-not-in-hyperplane}
Suppose we are in Setup \ref{main_setup}, $X$ is $d$-minimal, and $r(2)=2$. Consider a general triple of divisors $D, D_1, D_2 \in A$ and let $D_3=3D-D_1-D_2$. Then \[\bigcap_i \Span_{|3D|'} D_i = \emptyset.\]
\end{lemma}
\begin{proof}
Suppose the spaces $\Span_{|3D|'} D_i$ share a common point $P$. Consider a divisor $E$ such that $D, D_1, D_2, E$ is a general quadruple; then $D_1, D_2, D_3, 3D-E$ is a general quadruple as well and $P$ does not belong to $\Span_{|3D|'} E$. Consider the projection $\pi_E$ from $\Span_{|3D|'}E$. We have \[\pi_E(P) \in \pi_E(\Span_{|3D|'}D_i)=\Span_{|3D-E|} D_i.\] By the generality of the quadruple $D_1, D_2, D_3, 3D-E$, the lines $\Span_{|3D-E|} D_i$ do not share a point, contradiction. 
\end{proof}

We now prove the main result of this section.
\begin{theorem}\label{thm:main-debarre-fahlaoui}
Suppose we are in Setup \ref{main_setup} and $X$ is $d$-minimal. If $r(2) =2$, then $X$ is birational to a Debarre--Fahlaoui curve.
\end{theorem}

\begin{proof}
By  Lemma \ref{lem:dimg1}, we have \(\dim A = 1 \). By Lemma \ref{lem:rprime_3}, we have \(r'(3) = 5\). Let $D$ be a point of $A(k)$ achieving these generic values, so that the linear systems \(|nD|\) are basepoint-free for \(n \geqslant 2\), \(\dim |2D| = 2,\) and \(\dim |3D|' = 5\).  

Write $\varphi \colon X \to \pp^5$ for the morphism associated to $|3D|'$. 
We will now define a rational map from $\Sym^2 A$ to $\pp^5$, whose image contains $X$ in its closure.  Since $s'(3)=r'(3)-r(2)-1 = 2$, a general divisor  $[D_1] \in A$ has $2$-dimensional span.

   Given a general pair of divisors $[D_1], [D_2] \in A$, the spans $\Span_{|3D|'}D_1,$ $\Span_{|3D|'}D_2$ belong to a common hyperplane (by the abelian translate property \eqref{atp}). This means that for a general pair of divisors $[D_1], [D_2] \in A$, we must have $\dim \Span_{|3D|'} D_1 \cap \Span_{|3D|'} D_2 \geqslant 0$. Since by Lemma \ref{lem:no-extra-intersecions} there exists a pair of divisors with zero-dimensional intersection of spans, by semicontinuity we have in general $\dim \Span_{|3D|'} D_1 \cap \Span_{|3D|'} D_2 = 0$.
   
This yields a rational map
    \begin{align*}
        \psi \colon \Sym^2 A &\dashrightarrow \pp^5\\
        \qquad (D_1, D_2) & \mapsto \Span D_1  \cap \Span D_2.
    \end{align*}

If \(\Span D_1 \cap \Span D_2\) has dimension \(1\), then projecting from \(\Span(D_1 + D_2)\) yields a (geometric) degree \(d\) map \(X \to \pp^1\); in particular, by Setup \ref{main_setup}, the divisor \([3D - D_1 - D_2] \in A\) lies in a proper Zariski closed (dimension at most \(0\)) locus.  From this we observe 
\begin{multline} \label{positive-dim-intersections}
\text{For general \(D_1\), there exist finitely many lines \(\Sigma\) in \(\Span D_1\), such that for any \(P \not \in \Sigma\)},\\ \text{if \(D_2 \neq D_1\) and \(P \in \Span D_1 \cap \Span D_2\), then \(P = \Span D_1 \cap \Span D_2\).}
\end{multline}

    If the divisors \(D_1\) and \(D_2\) share a point, then the intersection of \(\Span D_1 \) and \(\Span D_2\) necessarily contains that point.  The closure of the image of \(\psi\) contains the image of \(X\) under \(\varphi\) since by Lemma \ref{lem:no-extra-intersecions} for a  general point \(P\) of \(X(\bar{k})\), there exist divisors \(D_1\) and \(D_2\) such that       
    \[ \Span D_1 \cap \Span D_2 = P.\]

    Our goal is to show that a general point $P \in X$ is contained in \emph{exactly} two divisors. Once we do so, we will have a natural map $X \to \Sym^2 A$, and we will then show via a simple argument that the image is indeed a Debarre--Fahlaoui curve. Fix a general divisor class $D_1$. We first analyze the image of the morphism $\eta: A \to \Span D_1$ that sends a divisor $D_2$ to $\Span D_2  \cap \Span D_1$. This map is 
    nondegenerate by Lemma \ref{lem:no-extra-intersecions}. 
 We will show that it must be the inclusion of \(A\) as a plane cubic curve by considering several cases based on the possible degrees of the image of \(\eta\).
\medskip

\noindent
\textbf{\boldmath Case 1: \(\deg\eta(A) =2\).} 
Consider divisors $D_2, D_3$ such that $D_1, D_2, D_3$ form a general triple. Let $P\colonequals \Span D_2\cap \Span D_1$ and $Q \colonequals \Span D_3 \cap \Span D_1$. Consider the divisor $D_4 \colonequals 3D-D_2-D_3$; since $D_1, D_2, D_3$ are a general triple we can assume that $D_4$ does not pass through $P$ or $Q$.
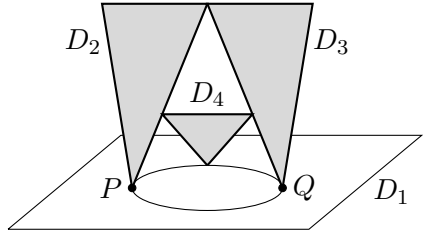
\begin{center}
\begin{tikzpicture}
\draw (0,0) -- (4,0) -- (5.5,1.25) -- (1.5,1.25) -- (0,0);
\draw (2.65,.55) ellipse (1cm and .3cm);
\filldraw (1.65, .55) circle[radius=0.05];
\draw (1.65, .55) node[left] {\(P\)};
\filldraw (3.65, .55) circle[radius=0.05];
\draw (3.65, .55) node[right] {\(Q\)};
\draw [thick, fill=gray!30] (1.65, .55) -- (1.25,3) -- (2.65,3) -- (1.65, .55);
\draw [thick, fill=gray!30] (3.65, .55) -- (3.65+.4,3) -- (2.65,3) -- (3.65, .55);
\draw [thick, fill=gray!30] (2.65,.85) -- (1.65+.4,.55+.98) -- (3.65-.4,.55+.98) --(2.65,.85);
\draw (5.1,.5) node{\(D_1\)};
\draw (1,2.5) node{\(D_2\)};
\draw (4.3,2.5) node{\(D_3\)};
\draw (2.65,1.8) node{\(D_4\)};
\end{tikzpicture}

\end{center}
On the other hand, $\Span(D_2, D_3)$ is a hyperplane that contains $\Span D_4$ and intersects $\Span D_1$ in the line $\overline{P Q}$.  Since \(\overline{PQ}\) meets \(\eta(A)\) only at \(P\) and \(Q\), this is a contradiction.
\medskip

\noindent
\textbf{\boldmath Case 2: \(\deg \eta(A)\geqslant 3\) and \(\eta\) multiple-to-one onto its image.} 

Choose a general line $\ell$ in $\Span D_1$ that intersects the image of $\eta$ in at least three smooth points $P_2, P_3, P_4$. 
Through each of the points $P_2, P_3, P_4$ passes at least two divisor-spans $\Span D_i, \Span D_i'$, $i=2,3,4$. 
Since the line \(\ell\) is general, the pair \(D_1, D_2\) is a general pair of divisors in \(A \times A\).  Hence, using the fact that \(\eta\) is nondegenerate, we can assume that the point where they meet is not in the finite set \(\Sigma\) guaranteed by \eqref{positive-dim-intersections} on either \(\Span D_1\) or \(\Span D_2\).  In particular
\[\Span D_1 \cap \Span D_2 = \Span D_1 \cap \Span D_2' = \Span D_2 \cap \Span D_2' = P_2.\]
By symmetry the same holds for \(i = 3\) and \(i=4\).

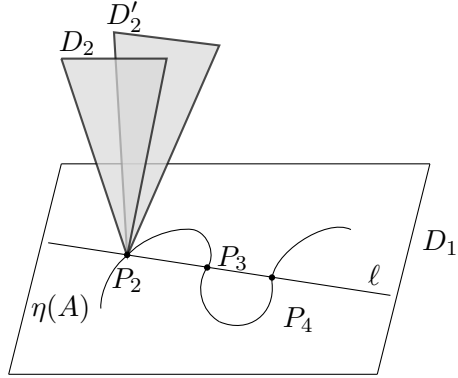
\begin{center}
\begin{tikzpicture}[scale=0.7]
	\begin{pgfonlayer}{nodelayer}
		\node [style=none] (0) at (-6, -1) {};
		\node [style=none] (1) at (-5, 3) {};
		\node [style=none] (2) at (2, 3) {};
		\node [style=none] (3) at (1, -1) {};
		\node [style=none] (4) at (-6, -1) {};
		\node [style=none] (5) at (-5.25, 1.5) {};
		\node [style=none] (6) at (1.25, 0.5) {};
		\node [style=none] (7) at (-4.25, 0.25) {};
		\node [style=none, label={below: $P_2$}] (8) at (-3.75, 1.25) {};
		\node [style=none] (9) at (-2.5, 1.75) {};
		\node [style=none] (10) at (-2.25, 1) {};
		\node [style=none] (11) at (-2, 0) {};
		\node [style=none] (12) at (-1, 0.75) {};
		\node [style=none] (13) at (0.5, 1.75) {};
		\node [style=none] (14) at (-5, 5) {};
		\node [style=none] (15) at (-3, 5) {};
		\node [style=none] (16) at (-4, 5.5) {};
		\node [style=none] (17) at (-2, 5.25) {};
		\node [style=none, label={above:$P_3$}] (18) at (-1.75, 0.75) {};
		\node [style=none, label={below:$P_4$}] (19) at (-0.5, 0.5) {};
		\filldraw (-3.75, 1.27) circle[radius=0.05];
		\filldraw (-1, .84) circle[radius=0.05];
		\filldraw (-2.23, 1.03) circle[radius=0.05];
		\draw (-4.7, 5.3) node{\(D_2\)};
		\draw (-3.8, 5.8) node{\(D_2'\)};
		\draw (7) node[left]{\(\eta(A)\)};
		\draw (6) node[above left]{\(\ell\)};
		\draw (2.2, 1.5) node {\(D_1\)};
	\end{pgfonlayer}
	\begin{pgfonlayer}{edgelayer}
		\draw (4.center) to (1.center);
		\draw (1.center) to (2.center);
		\draw (2.center) to (3.center);
		\draw (3.center) to (4.center);
		\draw (5.center) to (6.center);
		\draw [bend left, looseness=0.75] (7.center) to (8.center);
		\draw [bend left, looseness=0.75] (8.center) to (9.center);
		\draw [bend left=60] (9.center) to (10.center);
		\draw [bend right=45] (10.center) to (11.center);
		\draw [bend right=60, looseness=1.25] (11.center) to (12.center);
		\draw [bend left=60, looseness=0.75] (12.center) to (13.center);
		\draw[thick, fill=gray!30, opacity=0.7] (8.center) to (16.center) to (17.center) to (8.center);
		\draw[thick, fill=gray!30, , opacity=0.7] (14.center) to (8.center) to (15.center) to (14.center);
	\end{pgfonlayer}
\end{tikzpicture}

\end{center}

By Lemma \ref{lem:generically-not-in-hyperplane} applied to the triple $D, D_1, D_i$ we may choose \(D_i' \neq 3D - D_1 - D_i\).  
Since \(D_1, D_2\) is general, for every possible \(D_2'\) through \(P_2\), the divisor
$D_2-D_2'$ is not a $2$-torsion point on $A^0$, as we now explain.  Equivalently, for a general \(D_2\), and each of the finitely many possible \(2\)-torsion points \(T\) on \(A^0\) giving rise to \(D_2' = D_2 - T\), a general divisor span \(\Span D_1\) does not meet \(\Span D_2 \cap \Span D_2'\), which is clear since the the map \(\eta\) associated to \( D_2\) is nondegenerate.

We may further assume that the $15$ pairwise intersection points of $\Span D_i, \Span D_j'$ do not belong to $X$ (if \(i \neq j\) then \(D_i, D_j'\) are a general pair of divisors; if \(i=j\), then \(P_i\) is a general point on \(\eta(A)\) which contains only finitely many points of \(X\)).

Projection from \(\ell\) maps each \(D_i\) and \(D_i'\) for \(i=2,3,4\) to a line.
Since \(D_i' \neq 3D - D_1 - D_i\), we have that $\Span(D_i + D_i')$, for $i=2,3,4$, does not contain \(D_1\).  Further, since $\ell$ is a general line in \(\Span(D_1)\) through \(P_i\), we also have that $\Span(D_i + D_i')$ does not contain \(\ell\).
Since in addition \(\Span D_i\) and \(\Span D_i'\) meet only at \(P_i\), their images under projection from \(\ell\) are skew.  Moreover, since \(\Span D_i \cap \Span D_j'\) is a point off of \(\ell\) for \(i \neq j\), any two such lines meet.

The only such configuration of a triple of pairs of skew lines $\pi_\ell(\Span D_i), \pi_\ell(\Span D_i')$ is the configuration of edges of a tetrahedron. 

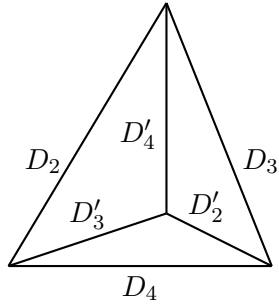
\begin{center}
\begin{tikzpicture}[scale=0.7]
	\begin{pgfonlayer}{nodelayer}
		\node [style=none] (0) at (-3, 0) {};
		\node [style=none] (1) at (2, 0) {};
		\node [style=none] (2) at (0, 5) {};
		\node [style=none] (3) at (0, 1) {};
		\node [left] (4) at (-1.8, 2) {$D_2$};
		\node [above] (5) at (0.75, 0.65) {$D_2'$};
		\node [below] (6) at (-0.5, 0) {$D_4$};
		\node [style=none] (7) at (0, 2.75) {};
		\node [left] (8) at (0, 2.5) {$D_4'$};
		\node [style=none] (9) at (-1.5, 0.5) {};
		\node [above] (10) at (-1.5, 0.5) {$D_3'$}; 
		\node [right] (11) at (1.25, 2) {$D_3$}; 
	\end{pgfonlayer}
	\begin{pgfonlayer}{edgelayer}
		\draw[thick] (0.center) to (2.center);
		\draw[thick] (2.center) to (3.center);
		\draw[thick] (3.center) to (0.center);
		\draw[thick] (3.center) to (1.center);
		\draw[thick] (1.center) to (2.center);
		\draw[thick] (0.center) to (1.center);
	\end{pgfonlayer}
\end{tikzpicture}
\end{center}
The triples of divisors corresponding to faces of the tetrahedron sum to $3D$, since they are coplanar. Summing the faces containing a shared edge and subtracting the two faces containing the opposite edge, we get $2(D_i-D_i')=0$ for all $i$, contradicting the generality assumption $D_2-D_2' \not\in A[2]$. 
\medskip

\noindent
\textbf{\boldmath Case 3: \(\deg \eta(A) \geqslant 4\) and \(\eta\) birational onto its image.} 
We proceed as before by choosing a general line $\ell \in \Span D_1$ and analyzing the projection from $\ell$. 
    Let $P_2, ..., P_n$, $n \geqslant 5$ be the points of $\ell \cap \eta(A)$. Since $\ell$ is general there is a unique divisor-span $\Span D_i$ through $P_i$. Projection from $\ell$ maps the $\Span(D_i)$ into a collection of lines in $\PP^3$ pairwise sharing points.  Since they cannot all belong to the same plane, by Lemma \ref{lem:linear-algebra} they have to share a point $P$. 
        
        Consider two divisors $D_i, D_j$ for $i,j>2$ and let $D_{ij}\colonequals 3D-D_i-D_j$ be the remaining divisor contained in the hyperplane $\Span(D_i, D_j)$.  Since \((D_1, D_i, D_j)\) is general, \(D_{ij} \neq D_1\); equivalently, \(D_1\) is not contained in $\Span(D_i, D_j)$.  In particular \(\Span(D_i, D_j) \cap \Span D_1 = \ell\), and so \(\Span D_{ij}\cap\Span(D_1) \in \ell\).  Thus $D_{ij}=D_k$ for some index $k$.   Projecting the configuration of lines $\pi_\ell(\Span(D_i))$ from $P$ gives a configuration of $n-1 \geqslant 4$ points $Q_2, ..., Q_n \in \PP^2$ with the following two properties:
        \begin{enumerate}
            \item\label{Sylvester-Gallai} for any two distinct points $Q_i, Q_j$ there exists a point $Q_k,$ $k \neq i,j$ collinear with $Q_i, Q_j$;
            \item no more than $3$ of $Q_i$ are collinear.
        \end{enumerate}
        Configurations of points satisfying Property (\ref{Sylvester-Gallai}) are known as Sylvester--Gallai configurations; see \cite[Theorems 3.1 -- 3.6]{kelly1973affine} for classification results for small values of $n$. In particular, either $n-1=9$ and the configuration is  Hesse configuration (and the points are a base locus for a pencil of cubic curves) or $n-1 \geqslant 12.$

We first show that for general choices of \(D_1\) and \(\ell\), the \(2\)-plane \(\Lambda \colonequals \pi_\ell^{-1}(P)\) does not meet \(X\).  Indeed, suppose to the contrary that \(\Lambda\) meets \(X\) in \(\lambda\) points for a general choice of \(D_1\) and \(\ell\).  Fixing \(D_1\) and varying \(\ell\), we obtain a map
\[(\PP^2)^\vee \dashrightarrow \Sym^\lambda X.\]
If this map is nonconstant, then \(\lambda \geqslant d+1\) since \(X\) is \(d\)-minimal.  Since \(\Span D_1\) meets \(X\) in \(d\) points, none of which are on \(\ell = \Lambda \cap \Span D_1\), we must have that \(\Span(D_1, \Lambda)\) meets \(X\) in at least \(2d + 1\) points.  If \(D_1\) and \(\ell\) are defined over the ground field, then so is the unique point \(P\), and hence so is the \(3\)-plane \(\Span(D_1, \Lambda)\).  Projection from this plane defines a degree at most \(d-1\) map to \(\pp^1\), contradicting \(d\)-minimality.

We may therefore assume that the plane \(\Lambda\) meets \(X\) in points independent of \(\ell\).  If it meets \(X\) in at least \(2\) distinct points in \(\pp^5\), then their span meets \(\Span D_1\) in a unique point, which does not lie on a general line \(\ell\).  Hence \(\Lambda\) meets the image of \(X\) in \(\pp^5\) in a unique point (possibly with multiplicity).  Varying \(D_1\), and noting that a \(d\)-minimal curve cannot have genus at most \(1\), we see that this unique point must also be independent of \(D_1\).  This is a contradiction, since a general pair of divisors \(D_1, D_2\) can be chosen so that their span \(\Span(D_1, D_2)\) misses any specific point.

        Consider the projection $\pi_\Lambda\colon X \to \PP^2$, and suppose that it factors as $X \to Y \to Y' \subset \PP^2$, where $X \to Y$ is a finite degree $s$ morphism of smooth curves and $Y \to Y'$ is birational. The plane curve $Y'$ has degree $3d/s$ and every one of the
        points $Q_2, ..., Q_n$ is a singular point of multiplicity at least $d/s$. Moreover, the point $Q_1\colonequals \pi_\Lambda(\Span D_1))$ is also a singular point of $Y'$ of multiplicity at least $d/s$.
        
        Suppose $n-1=9$. Then the configuration of points $Q_2, ..., Q_{10}$ is a Hesse configuration, in particular there is a pencil of cubics through $Q_2,..., Q_{10}$. We may therefore choose $Q$ to be a cubic though $Q_2, ..., Q_{10}$ and $Q_1$. The curve $Y'$ is not a cubic, since it has at least $10$ singular points, and so $Y' \cap Q$ is a finite scheme. However $Q$ intersects $Y'$ in at least $10$ points of multiplicity $d/s$, thus the total multiplicity of the intersection is at least $10d/s > 9d/s=\deg Q \deg Y'$, contradiction.
        
        Suppose $n-1 \geqslant 12$. Then the geometric genus of $Y'$ is at most 
        \[g_{Y'} \leqslant \frac{(3d/s - 1)(3d/s -2 )}{2} - 13 \frac{d/s(d/s-1)}{2}=1+ 2 \frac{d}{s} - 2\frac{d^2}{s^2}.\]
        Since $Y'$ has at least $13$ singular points, the degree of $Y'$ is at least $5$, and so $g_{Y'}<0$, contradiction.

\medskip

Thus the map $\eta\colon A \to \Span D_1$ is an isomorphism onto a plane cubic curve. This means that for every point $P \in D_1$ there is exactly one divisor $D_2 \neq D_1$, with $[D_2] \in A$ that contains $P$. Since $D_1$ was an arbitrary general divisor, we conclude that for a general point $P \in X$ there exist exactly two divisors $D_1,D_2$ from $A$ that contain $P$. Therefore we can define a birational morphism $\mu\colon X \to \Sym^2 A$ that sends a point $P \in X$ to the unique pair of divisors $(D_1, D_2) \in \Sym^2 A$ that contain $P$. We claim that $\mu$ is the birational equivalence of $X$ with a Debarre--Fahlaoui curve that we seek. To do this we need to identify the numerical class of $\mu(X)$ on $\Sym^2 A$. 

For $x \in A$, recall that \(H_x\) is the set of all pairs of divisors in \(\Sym^2 A\) that contain \(x\).  Hence
 the intersection $H_x \cap \mu(X)$ is supported on the points \(\mu(\supp(x))\).
 Since a general divisor \(x \in A\) is a degree \(d\) point (and hence a single monodromy orbit), $H_x \cap \mu(X)$ is a multiple of \(\mu(\supp(x))\). Consider a general point \((x, x')\) on \(\mu(X)\).
 Since the divisors $H_x$, and $H_{x'}$ for \(x \neq x'\) intersect transversely at $(x,x') \in \Sym^2 A$, the intersections \(H_x \cap \mu(X)\) and \(H_{x'} \cap \mu(X)\) cannot both be nontransverse at \((x, x')\).
 Hence the generic intersection $H_x \cap \mu(X)$ cannot consist of multiple points.  In other words, for a general $x \in A$ the intersection $H_x \cap \mu(X) \subset \Sym^2 A$ is smooth and \(H_x \cap \mu(X) = \mu(\supp(x))\), so $[\mu(X)]\cdot H=d$. Therefore, numerically, $[\mu(X)]=aH+(d-a)F$. Since $[\mu(X)]$ is effective, by the description of the effective cone we have $a \geqslant 0$ and $2d-a \geqslant 0$. The fibers of the addition map $\Sym^2 A \to A$ have numerical class $F$, and since $X$ does not admit maps of degree less than $d$ to $A$, we conclude that $d<[\mu(X)]\cdot F=a$. Thus $[\mu(X)]=(d+m)H-mF$ for some $m$ between $1$ and $d$ as claimed.
\end{proof}

\begin{remark}\label{rem:gonality-2}
As observed in Remark \ref{rem:gonality-1}, the geometric gonality of a curve with $\min(\dendegs(X/k))=d$ which is not a degree $d$ cover of an elliptic curve is at most $2d-1$. Theorem \ref{thm:main-debarre-fahlaoui} implies that if, in addition, such curves are not Debarre--Fahlaoui curves, then their geometric gonality is at most $2d-2$.
\end{remark}

\section{Applications and Extensions}\label{sec:Applications}
\subsection{Low degree points on projective curves}\label{sec:special-curves}
Our main strategy can be applied to study low degree points on  ``special'' curves. The geometry of configurations of divisor spans as summarized in Section \ref{sec:construction-summary} can be used to estimate the dimensions of various linear systems from below. Combining this with Castelnuovo's bound yields a bound on the genus of the curve. We now recall Castelnuovo's theorem \cite[Chapter III, page 116]{ACGH}. Given positive integers $\delta , n$, write
\[\delta-1 = m (n-1) + \varepsilon,\] for  integers $m$ and $0 \leqslant \varepsilon < n-1.$ Then the genus of a nondegenerate curve of degree $\delta$ in $\PP^n$ is bounded by 
\begin{equation}\label{Castelnuovo-function}
    \pi(\delta,n) = \frac{m(m-1)}{2}(n-1) + m\varepsilon.
\end{equation} For fixed $n$ and large $\delta$, the genus bound \(\pi(\delta, n)\) is roughly $\delta^2/(2n-2)$. 
\begin{theorem}\label{thms:special-curves}
Suppose $X \subset \PP^r$ is an irreducible (possibly singular) curve of degree $e$ and genus $g$. Suppose $X$ has infinitely many points of degree $d$ not contained in hyperplanes of $\PP^r$. Then 
\[g \leqslant \pi(e+2d, 2r+1).\]
\end{theorem}
\begin{proof}
By the Mordell--Lang Conjecture, as explained in Section \ref{sec:Mordell-Lang},
for all but finitely many degree $d$ points $D$ on $X$, either $D$ moves in a pencil, or the class of $D$ in $W_dX$ belongs to a translate of an abelian subvariety in $\Pic^d_X$. In either of those cases, the class $2D$ is basepoint-free.

Let $[H]$ denote the divisor class corresponding to the embedding $X \subset \PP^r$, choose a degree $d$ point $D$ for which $2D$ is basepoint-free and such that $D$ is not contained in divisors from $|H|$. Consider the linear system $|2D+H|$. Since $2D$ is basepoint-free, for a divisor $H' \in |H|$ we have $X \cap \Span_{|2D+H|} H'=H'$. Suppose the dimension of $\Span_{|2D + H|} D$ is equal to $s$ and choose a set $S \subset D$ of size $s+1$ such that $\Span_{|2D+H|} S = \Span_{|2D+H|} D$. If $s<r$, then there exists a divisor $H' \in |H|$ that contains $S$. Then $\Span_{|2D+H|} H'\supset \Span_{|2D+H|} S = \Span_{|2D+H|} D$, and so the points of $D$ are contained in \(X \cap \Span_{|2H + D|}H'=H'\), contradicting our assumption. Therefore $s\geqslant r$. Since the projection from $\Span_{|2D+H|} D$ maps to a space of dimension at least $r$, we have $\dim |2D+H| \geqslant 2r+1.$ By Castelnuovo's theorem applied to the embedding $|2D+H|$, the genus of $X$ is at most $\pi(e+2d, 2r+1)$.
\end{proof}

\subsection{Genus estimate for non-Debarre--Fahlaoui curves}
The argument of Theorem \ref{thm:weak-genus-bound} with the added assumption
 $r(2)>2$  gives a better genus bound for currves that satisfy \eqref{dagger}. Together Theorems \ref{thm:weak-genus-bound}, \ref{thm:main-debarre-fahlaoui}, \ref{2D-is-birational} and \ref{thm:genus-non-DF} below yield Theorem \ref{main-finer} from the introduction.
\begin{theorem}\label{thm:genus-non-DF}
Suppose $X$ is a $d$-minimal curve with $r(2)>2$. Suppose that condition \eqref{dagger} holds. Then the genus $g$ of $X$ satisfies \[g \leqslant \frac{(d-1)(d-2)}{2}+2\].
\end{theorem}
\begin{proof}
The argument is identical to the proof of Theorem \ref{thm:weak-genus-bound}. We have $r(2) \geqslant 3$, $s(2)\geqslant 2$ by assumption and $s(n) \geqslant \min(d-1, s(n-1)+1)$ by Proposition \ref{extra-intersections}. Therefore for $n\leqslant d - 1$ we have \[r(n) \geqslant \frac{(n+1)(n+2)}{2}-3.\]
By Castelnuovo's theorem applied to the linear system $|(d-1)D|$ we get the desired genus bound.
\end{proof}
\subsection{Classification results for low values of \texorpdfstring{$d$}{d}}\label{sec:small-airr}

We now summarize what the main results say about curves with small minimum density degree.

\begin{proposition}\label{prop:small-numbers}
The following Table \ref{tab:classification} summarizes the classification of curves \(X\) of genus \(g\) with small values of \(\min(\dendegs(X/k))\) or \(\min(\potdendegs(X/k))\).  We use the following shorthand:
\begin{itemize}
\item ``covers": a degree \(d\) cover of \(\pp^1\) or a (positive rank) elliptic curve
\item ``DF": a normalization of a Debarre--Fahaloui curve
\end{itemize}
\end{proposition}

\begin{proof}

If \(X\) is not \(d\)-minimal, then it is a cover of an \(s\)-minimal curve for some \(s \mid d\).  Since there are no \(2\)-minimal curves by Theorem \ref{thm:no_2_min}, and in our cases $d \leqslant 5$, we can assume that \(X\) is either \(d\)-minimal or a cover of a \(1\)-minimal curve (i.e., \(\pp^1\) or an elliptic curve of positive rank).   From now on, we assume that \(X\) is \(d\)-minimal.  If \(X\) is not a Debarre--Fahlaoui curve, but satisfies the condition \eqref{dagger}, then by Theorem \ref{thm:genus-non-DF}, the genus of \(X\) is at most \((d-1)(d-2)/2 + 2\). If \eqref{dagger} does not hold, then the genus bound from Theorem \ref{thm:weak-genus-bound} applies. Finally, when $d=5$ the genus of $X$ is bounded by the Castelnuovo function $\pi(20, 12)=8$ by Lemma \ref{genus-fix-2}.

Any curve of genus \(g\) has geometric gonality at most \(\lfloor (g+3)/2 \rfloor\) and gonality at most \(2g - 2\).  Combining this with the genus bounds described above gives the result.
\end{proof}

\begin{table}[ht!]
\scalebox{0.95}{
\begin{tabular}{c | c | c | c| c}
\(d\) & \(2\) & \(3\) & \(4\) & \(5\) \\ \hline
\(\min(\potdendegs) = d\) & covers & covers & covers + DF & covers + DF\\
\(\min(\dendegs) = d\) & covers & covers  + DF + \(g=3\)& covers + DF + \(g=4,5\) & covers + DF + \(g = 5, 6, 7, 8\)\\
\end{tabular}}
\vspace{5pt}
\caption{\label{tab:classification}The classification of curves with small minimum density degree}
\end{table}

In the case \(d=3\) and \(g=3\), a \(3\)-minimal curve \(X\) cannot be hyperelliptic, since any conic with a degree \(3\) point is isomorphic to \(\pp^1\).
Hence a \(3\)-minimal curve \(X\) of genus \(3\) is isomorphic to a smooth plane quartic.  Since the gonality of a plane quartic is \(3\) if and only if it has a rational point, we see that \(X\) must be pointless. In this case, \(\Sym^3 X(k) = \Pic^3_X(k)\), and so the Jacobian of \(X\) must have positive rank.  Conversely, any such curve with a single degree \(3\) point which is not a degree \(3\) cover of an elliptic curve (i.e., with simple Jacobian) is \(3\)-minimal.
Combining this with Proposition \ref{prop:small-numbers} proves Theorem \ref{main-low-degrees}.

\section{Questions and Problems}\label{sec:questions}

\subsection{Geometric questions}\label{sec:geometric-questions}

All of the questions we consider have a geometric analogue, that applies to curves $X$ over any field $k$, and concerns the existence of abelian translates in $W_d \bar{X}$. The resulting geometric questions are usually slightly easier then the arithmetic ones.

Given a curve \(X\) over any field \(k\), 
the union of all positive-dimensional abelian translates in \(W_d \bar{X}\) is the \defi{Kawamata--Ueno locus} \(\Ueno(W_d \bar{X})\). 
When \(k\) is a number field, the Mordell--Lang conjecture implies that 
\[\min(\potdendegs(X/k)) = \min(\gon(\bar{X}), \min(d : \Ueno(W_d \bar{X}) \neq \emptyset)).\]
However, this more general definition makes sense for complex curves.  It is thus natural to define the locus \(Z_d(g)\) of curves \([X ]\in M_g(\cc)\) with \(\min(\gon(\bar{X}), \min(d : \Ueno(W_d \bar{X}) \neq \emptyset)) = d\).  For \(d \geqslant \lfloor (g+3)/2 \rfloor\), the locus \(Z_d(g)\) is all of \(M_g(\cc)\).  The general theory of Kawamata--Ueno loci \cite[Theorem 1.2]{McQuillan1994} implies that $Z_d(g)$ is the set of complex points a subvariety of $M_g$.  For lower values of $d$ it is interesting to study the ubiquity of curves in \(Z_d(g)\).

\begin{question}
For \(d < \lfloor (g+3)/2 \rfloor\), what is the codimension of \(Z_d(g)\) in \(M_g\)?
\end{question}

\medskip

A geometrically \(d\)-minimal curve $X$ is a curve in \(Z_d(g(X))\) for which there does not exist a degree $s \geqslant 2$ covering of curves $X \to Y$ with $Y \in Z_{d/s}(g(Y))$. 
The next specific case in which we don't know the classification of geometrically \(d\)-minimal curves is \(d=6\) and \(g=11\).  
\begin{question}
Do there exist geometrically \(6\)-minimal curves of genus \(11\)?
\end{question}
Abramovich and Harris claimed \cite{Abramovich-Harris1991}[Theorem 2] that the genus of a geometrically $d$-minimal curve is bounded by $d(d-1)/2$; however the proof presented there is incomplete. The key to obtaining this bound is the inequality $s(n) \geqslant \min(s(n-1)+1, d-1)$, which we only prove in Corollary \ref{cor:s(n)-grows} under the assumption \eqref{dagger}. It seems likely that the genus bound holds without additional assumptions.

\begin{question}\label{que:no-daggers}
Do there exist geometrically $d$-minimal curves with genus larger than $d(d-1)/2$?
\end{question}

In the geometric situation, it is natural to treat $a=\dim A$ as an extra parameter together with \(d\) and \(g\). The key to the proof of the Main Theorem is Proposition \ref{extra-intersections} concerning the difference between the dimensions \(s(n)\) and \(s(n-1)\) of divisor spans in \(|nD|\) and \(|(n-1)D|\). It seems likely that such a bound can be strengthened to depend on $a$.

\begin{question}\label{que:geometric-dim-A}
If \(\dim A  = a\), is it true that when \(s(n) \leqslant d-2\), we have
\(s(n) - s(n-1) \geqslant a\) for \(n \geqslant 3\)?
\end{question}

If Question \ref{que:geometric-dim-A} has a positive answer, then one can obtained significant improvements on the genus bound for $d$-minimal curves with $a>1$. In particular it would imply that the $a=2$-family constructed by Debarre and Fahlaoui \cite{Debarre-Fahlaoui1993} achieves the largest possible genus. Such an estimate is claimed in \cite{Abramovich-Harris1991}, but the proof has a gap (as remarked in \cite{Debarre-Fahlaoui1993}).

\medskip

The problem of classifying curves in $Z_d(g)$ is interesting over any field $k$. The results of this paper use the assumption $\Char k = 0$ in a few places. For instance, Kani's version of de Franchis theorem (Theorem \ref{thm:kani}) requires a separability assumption in positive characteristic, the classification of small Sylvester--Gallai configurations used in the proof of Theorem \ref{thm:main-debarre-fahlaoui} is more complicated when $\Char k = 2,3$, and the monodromy argument in Lemma \ref{lem:dimg1} only works in characteristic zero. 

\begin{question}
Do the geometric analogues of our main Theorems \ref{thm:weak-genus-bound}, \ref{thm:main-debarre-fahlaoui} hold in positive characteristic?
\end{question}

\subsection{Arithmetic Questions}\label{sec:arithmetic-questions}

The smallest cases for which some questions are still left open concern $d$-minimal curves with $d=3$ and $g=3,4$. In both of these cases we expect that $3$-minimal curves exist. The case $d=g=3$ is discussed in Problem \ref{d=g-minimal}. In the case $d=3, g=4$, Proposition \ref{prop:small-numbers} implies that such a curve is a Debarre--Fahlaoui curve of class $4H-F$. However, a genus $4$ curve $X$ over $\kbar$ generically admits two maps of degree $3$ to $\PP^1$: the canonical embedding realizes $X$ as a complete intersection of a quadric and a cubic, and projections from rulings on the quadric give a pair of $g^1_3$'s. We expect that there exist Debarre--Fahlaoui curves in class \(4H-F\) for which these two maps are Galois conjugate, and that such Debarre--Fahlaoui curves are $3$-minimal.

It is in principle possible to verify this claim (if true) by exhibiting a specific curve on $\Sym^2 A$ and checking that the unique quadric containing the canonical curve is nonsplit (and independently verifying that it is not a triple cover of an elliptic curve), but such a computation is nontrivial in practice. We thus leave this question as a problem.

\begin{problem}
Show that a general curve in numerical class $4H-F$ on $\Sym^2 A$ is $3$-minimal.
\end{problem}

There is another natural source of low degree points on curves of genus $g$, as we now describe. Consider a general (in a non-technical sense) genus $g$ curve $X$ over a number field equipped with a degree $g$ point. The abelian variety $\Pic^g_X$ may have positive rank, at the same time it appears that there is no clear reason for $X$ to have maps to other curves or low gonality. If this is the case, then $\Sym^g X$, which is birational to $\Pic^g_X$, will have an infinite family of rational points. While we expect such curves to be abundant, we do not know if examples can be proved to exist, and thus leave this as a problem.
\begin{problem}\label{d=g-minimal}
Show that for every $d\geqslant 3$ there exists a $d$-minimal curve of genus $d.$
\end{problem}
The problem for $d=3$ (i.e., smooth plane quartics) is already interesting and should be computationally feasible. 
\medskip

Theorem \ref{thm:weak-genus-bound} shows that, under condition \eqref{dagger}, the genus of a $d$-minimal curve is bounded by $\frac{d(d-1)}{2}+1$; curiously, this number is exactly the (maximal) genus of a Debarre-Fahlaoui curve, and in Theorem \ref{thm:main-debarre-fahlaoui} we explained this conincidence by identifying curves with $r(2)=2$ with Debarre--Fahlaoui curves. Question \ref{que:geometric-dim-A} predicts a similar situation for $\dim A =2$:  the maximal genus for such a curve is $d^2/4 +1$, which is exactly the genus of $\dim A = 2$ examples constructed in \cite{Debarre-Fahlaoui1993}. We hope the this can also be explained in terms of geometry of configurations. 
\begin{question}
Is it true that a $d$-minimal curve with $\dim A = 2$ and $r(2)=3$ is birational to one of the curves constructed in \cite{Debarre-Fahlaoui1993}?
\end{question}

It would be very interesting to obtain better results for special curves; in particular, we do not expect Theorem \ref{thms:special-curves} to be close to optimal. One way to test optimality of such a genus bound is to compare it to known results in low dimension.
\begin{problem}
Suppose $X$ is a curve equipped with a $g^r_e$ linear system. Show that for a certain function $g(r,e,d)$ the following holds: if the genus of $X$ is larger than $g(r,e,d)$, and $X$ has infinitely many points of degree $d$, then all but finitely many of those points are contained in the divisors from $g^r_e$. Can the function $g(r,e,d)$ be such that the value of $g(1,e,d)$ implies Vojta's estimate \cite{Vojta1992generalization} for low degree covers of $\PP^1$, and the value of $g(2,e,d)$ implies the Debarre--Klassen theorem \cite{Debarre-Klassen1994} on smooth plane curves?
\end{problem}
\bibliographystyle{alpha}
\bibliography{bibliography}
\end{document}